\documentclass[letterpaper,10pt,english]{article}
\usepackage{babel}
\usepackage[latin1]{inputenc}
\usepackage{amsmath,amsthm,amssymb,amsfonts,indentfirst}
\usepackage[title]{appendix}

\newtheorem{theo}{Theorem}[section]

\newtheorem{lem}[theo]{Lemma}
\newtheorem{prop}[theo]{Proposition}
\newtheorem{claim}[theo]{Claim}
\theoremstyle{remark}

\numberwithin{equation}{section}
\numberwithin{theo}{section}

\newcommand\R{\text{I\!R}}
\newcommand\N{\text{I\!N}}

\newcommand\al{\alpha}
\newcommand\e{\varepsilon}
\newcommand\de{\delta}
\newcommand\be{\beta}

\newcommand{\Om}{\Omega}

\newcommand{\fr}{\partial}
\newcommand{\equ}{\eqref}
\newcommand{\grad}{\nabla}
\newcommand{\ml}{\mathcal}

\newcommand{\DD}{\ml{D}}
\newcommand{\sm}{\setminus}
\newcommand{\la}{\lambda}
\newcommand{\st}{such that }
\newcommand{\dem}{\bf Proof:}

\newcommand\lap{\Delta}
\newcommand\lab{\Delta_g}

\newcommand\ti{\tilde}

\newcommand{\lf}{\left}
\newcommand{\rg}{\right}
\newcommand\p{^}
\newcommand\ds{\displaystyle}

\newcommand{\bebs}{\begin{equation*}\begin{split}}
\newcommand{\ee}{\end{equation*}}
\newcommand{\esp}{\end{split}}

\DeclareMathAlphabet{\mathpzc}{OT1}{pcz}{m}{it}

\begin{document}
\title{Bubbling solutions for Moser-Trudinger type equations\\ on compact Riemann surfaces}
\author{Pablo
Figueroa\thanks{Escuela de Educaci\'on de Matem\'atica e Inform\'atica Educativa, Universidad Cat\'olica Silva Henr\'iquez, General Jofr\'e 462, Santiago, Chile. E-mail: pfigueroas@ucsh.cl. Author
supported by grants Fondecyt Postdoctorado 3120039 and Fondecyt Iniciaci\'on 11130517, Chile.}
\qquad and \qquad Monica Musso\thanks{ Departamento de
Matem\'atica, Pontificia Universidad Cat\'olica de Chile, Avenida
Vicu\~na Mackenna 4860, Santiago, Chile. E-mail:
mmusso@mat.uc.cl. Author supported by grants FONDECYT Grant 1160135 and Millennium Nucleus Center for Analysis of PDE, NC130017 }\\
}
\date{\today}
\maketitle

\begin{abstract}
\noindent We study an elliptic equation related to the
Moser-Trudinger inequality on a compact Riemann surface $(S,g)$,
$$\ds \Delta_g u+\la \lf(ue^{u^2}-{1\over |S|}  \int_S ue^{u^2} dv_g\rg)=0,\quad\text{in
$S$},\qquad
\ds\int_S u\,dv_g=0,
$$
where $\la>0$ is a small parameter,  $|S|$ is the area of $S$, $\Delta_g$ is the Laplace-Beltrami operator and $dv_g$ is the area element. Given any integer $k\ge 1$, under general conditions on $S$ we find a bubbling solution $u_\la$ which blows up at exactly $k$ points in $S$, as $\la\to0$. When
$S$ is a flat two-torus in rectangular form, we find that either seven or nine families of such solutions do exist for $k=2$. In particular, in any square flat two-torus actually nine families of bubbling solutions with two bubbling points do exist. If $S$ is a Riemann surface with non-constant Robin's function then at least two bubbling solutions with $k=1$ exists.
\end{abstract}

\\[0.1cm]


\section{Introduction}
\noindent  Let $(S,g)$ be a compact, orientable Riemann surface. We denote by $|S|$  the area of $S$,
$\Delta_g$ the Laplace-Beltrami operator on $S$ and $dv_g$ the area
element. This paper is devoted to the construction of solutions to the problem
\begin{equation}\label{mt}
\begin{cases}
\ds \Delta_g u+\la \lf(ue^{u^2}-{1\over |S|}  \int_S ue^{u^2} dv_g\rg)=0,\quad\text{ in
$S$},\\
\\[-0.3cm]
\ds\int_S u\, dv_g=0,
\end{cases}
\end{equation}
for any values of the small parameter $\la>0$. These solutions turn out to blow-up, as the parameter $\la \to 0^+$ at very specific points of $S$.

Problem \eqref{mt} is related to  the Trudinger-Moser inequality \cite{moser}
over a compact Riemann surface ($S$, $g$), which can be stated as follows
$$\sup\lf\{\int_S e^{4\pi u^2}\,dv_g\,:\, u\in H^{1}(S),\,\int_Su\,dv_g=0\text{ and }\int_S|\grad u|^2\,dv_g=1\rg\}<+\infty.$$
This type of inequality was first proved in \cite{Fo} on compact Riemannian manifolds of any dimension $n$. When the dimension is two, this inequality was proved  in \cite{Li0}  on manifolds with and without boundary,  and the existence of extremal functions was established. We refer also to \cite{Li1,Li2,YYang,YYang1} for related results and generalizations.

 It is simple to see that critical points of the above constrained variational problem satisfy, after a simple scaling, an equation of the form \eqref{mt}. Our purpose then is to study the existence of solutions to \eqref{mt} for $\la$ positive and small and to describe their asymptotic behavior as $\la\to0^+$.

Weak solutions of \eqref{mt} are critical points of the following energy functional
\begin{equation}\label{energy}
J_\la (u) = {1\over 2} \int_S | \nabla u |_g^2\,dv_g -{\la \over 2}
\int_S e^{u^2}\,dv_g , \quad u\in \bar H,
\end{equation}
where $\bar H =\{u \in H^1(S): \int_S u dv_g=0\}$, which corresponds to the free energy associated to the critical Trudinger embedding in the sense of Orlicz spaces \cite{poho, trudinger, yudo}
$$\bar H\ni u\mapsto e^{u^2}\in L^p(S)\ \ \forall p\ge 1.$$
The energy functional \eqref{energy} is thus well defined and it has a Mountain Pass geometric structure. Nevertheless, it is characterized by lack of compactness, which makes it impossible to search for critical points of \eqref{energy} using the classical tools of the Calculus of Variations or of the Critical Point Theory.
Indeed, loss of compactness translates into the presence of non-convergent Palais-Smale  sequences for the corresponding functional and space of functions.

\medskip
To better understand this, let us consider the flat case, namely, when $S \equiv \Om\subset \R^2$ is a bounded domain. The Trudinger-Moser inequality concerns the limiting case $p=2$ of the Sobolev embeddings $W^{1,p}(\Om)\subset L^{2p\over 2-p}(\Om)$. It states that there exists $C_2 >0$  such that
$$\sup_{\lf\{v\in W_0^{1,2}(\Om),\ \|\grad v\|_{L^2(\Om)}=1\rg\} }\int_\Om e^{\alpha |v|^2}\, dx\; \,\begin{cases}
\le C_2|\Om|,&\text{ if }\al\le \al_2\\
=+\infty,&\text{ if }\al> \al_2
\end{cases} ,$$
where $|\Om|$ is the area of $\Om$
and $\alpha_2=4\pi$. After a simple scaling, critical points of the above
constrained variational problem satisfy
the equation
\begin{equation}\label{tmce}
\Delta u+\la ue^{u^2}=0\quad \text{in
$\Om$}, \quad u=0 \quad \text{on $\fr\Om$}
,
\end{equation}
where $\la>0$, whose associated energy functional is
\begin{equation*}
I_\la (u) = {1\over 2} \int_\Om | \nabla u |^2 -{\la \over 2}
\int_\Om e^{u^2}, \quad u\in H_0^1(\Om).
\end{equation*}
 For the functional $I_\la$ a precise classification of all Palais-Smale sequences  does not seem possible after the results in \cite{adipas}.  Some information is available for sequences of solutions to \eqref{tmce}, thanks to the result in \cite{Dru}, that states

\medskip
{\it Assume that $u_n$ solves problem \eqref{tmce} for $\la=\la_n$, with $I_{\la_n}(u_n)$ bounded and $\la_n\to 0$  as $n\to+\infty$. Then, passing to a subsequence if necessary, there is an integer $k\ge 0$ such that as $n\to+\infty$}
\begin{equation}\label{jlnun}
I_{\la_n}(u_n)=2\pi k +o(1).
\end{equation}

\smallskip
\noindent A more precise characterization of the sequence of solutions $(u_n)_n$ is known when $k=1$, see \cite{AdDru}: for all large $n$, the
solution $u_n$ has only one isolated maximum,  whose value diverges to $+\infty$ as $\la_n\to 0$, which is attained
around a very specific point $x_0\in \Omega$. In fact, $x_0$  is  a critical point
of Robin's function, defined as $x\mapsto H_\Om(x,x)$, where $H_\Om$ is the regular part of the corresponding Green's function for the homogeneous Dirichlet  problem in $\Om$.

Concerning existence of solutions to \eqref{tmce} satisfying \equ{jlnun},
 in \cite{adipas}, it is proven  that there is a $\la_0>0$ such that a
solution to \equ{tmce} exists whenever $0<\la <\la_0$ (this is in fact
true for a larger class of nonlinearities with {\em critical
exponential growth}). By construction this solution falls into the bubbling category \equ{jlnun} with $k=1$ as $\la
\to 0$. If $\Omega $ has a sufficiently small hole, Struwe in \cite{struwe2} built a solution taking advantage of the presence of topology. This solution exists for a class of
nonlinearities, perturbation of the Trudinger-Moser one, that also
include $\la u e^{u^2-u}$ for which no solution exists for small
$\la$, in a disk, see \cite{adi2, FR}. It is reasonable to believe
that the construction of Struwe in reality produces a second
solution of equation \equ{tmce}, but this is not known yet. Similar results to \cite{adistruwe,Dru} on compact Riemann surfaces are obtained in \cite{YYang}.

\medskip
 In \cite{dmr1} authors addressed the existence of
bubbling solutions for \eqref{tmce} as $\la\to 0$ when $\Om$ is not contractible to a point, and $k$ in \eqref{jlnun} is any integer number. They provide sufficient conditions for the existence of solutions to \eqref{tmce} for small $\la$, which satisfy the bubbling condition \eqref{jlnun} and give a precise characterization of its bubbling location. In particular, they show
that if $\Om$ has a hole of any size, namely, $\Om$ is not simple connected then at least one of such a
solutions exists with $k=2$  and if $\Om$ has $d\ge1$ holes, then
$d+1$ solutions with $k=1$ exist.
\medskip

The question we address in this paper is whether it is possible to construct a family of solutions $u_\la $ to problem \eqref{mt}, for any $\la >0$ small, whose energy $J_\la (u_\la )$, defined in \eqref{energy}, is quantized in the sense of \eqref{jlnun}, and whose asymptotic behavior resembles a bubbling phenomena at points, for $\la \to 0$.

\medskip
In order to state our general result, let us introduce some notations. For a given Riemann surface $(S,g)$, we introduce the Green's function $G(x,p)$ with pole at $p \in S$ as the solution of
\begin{equation} \label{green}
\left\{ \begin{array}{ll} \ds-\Delta_g G(\cdot,p)= \delta_{p}-\frac{1}{|S|} &\text{in $S$}\\[0.4cm]
\ds\int_S G(x,p)\, dv_g=0. &
\end{array} \right.
\end{equation}
Let $k\geq 1$ be an integer, $\xi_1,\xi_2,\ldots, \xi_k \in S$ be $k$ distinct points and $m_1,m_2,\ldots, m_k$ be $k$ positive numbers.
We define the following functional
\begin{equation}\label{fik}
\begin{split}
\psi_k(\xi,m)=&\,(\log 16-2)\sum_{j=1}^km_j^2+\sum_{j=1}^km_j^2\log m_j^2-4\pi\sum_{j=1}^km_j^2H(\xi_j,\xi_j)\\
&\,-4\pi\sum_{i=1}^k\sum_{j=1,j\ne i}^km_i m_jG(\xi_i,\xi_j),
\end{split}
\end{equation}
where $\xi=(\xi_1,\dots,\xi_k)$ and $m=(m_1,\dots,m_k)$. Here, $G$ is the Green's function  for the Laplace-Beltrami operator on $S$ given by \eqref{green} and  $H$ is its regular part.  Let us consider an open set ${\mathcal D}$ compactly contained in
the domain of the functional $\psi_k$, namely
$$
\bar {\mathcal D} \subset  \{(\xi, m)\in S^k \times \R_+^k \ \mid \
\xi_i\ne \xi_j \ \forall \, i\ne j\ \} .
$$
We say that $\psi_k$ has a {\em stable critical point situation}
if there exists a $\delta>0$ such that for any $g\in C^1(\bar
{\mathcal D} )$ with $\|g\|_{C^1(\bar {\mathcal D} )} < \delta$, the
perturbed functional $\psi_k+ g $ has a critical point in
${\mathcal D}$.

\medskip
We can now state our general result.

\begin{theo} \label{teo2}
Let $(S,g)$ be a compact, orientable Riemann surface.
Let $k\ge 1$ and assume that there is an open set $\DD$ where
$\psi_k$ has a stable critical point situation. Then, for all
small $\la>0$  there exists a family of solutions $u_\la $ of problem $\equ{mt}$
such that as $\la\to 0$
\begin{equation}\label{uno}  {1\over 2} \int_S | \nabla u_\la |^2_g\, dv_g -{\la \over 2}  \int_S e^{u_\la^2}\, dv_g \ =\  2k\pi \, +\, O(\la).
\end{equation}
Moreover, there exists $(\xi_\la, m_\la)\in \DD$, with $\xi_\la=\big(\xi_\la^{(1)},\dots,\xi_\la^{(k)}\big)$ and $m_\la=\big(m_\la^{(1)},\dots,m_\la^{(k)}\big)$ such that, passing to a subsequence, $(\xi_\la,m_\la)\to (\xi_0,m_0)$ with $\nabla \psi_k (\xi_0, m_0)=0$ and
\begin{equation}\label{due}
u_\la (x) \ =\ \sqrt{\la} \, \left (\, 8\pi \sum_{j=1}^k  m_{\la}^{(j)}G(x,\xi_{\la}^{(j)}) \ + \ O(\la) \, \right )
\end{equation}
as $\la\to 0$, uniformly on compact subsets of $S\setminus \{\xi_1,\dots, \xi_k\}$.

\end{theo}

\medskip
Concrete examples of surfaces $S$ on which problem \eqref{mt} has solutions satisfying \eqref{uno}-\eqref{due} depends on the possibility to ensure the existence of special critical points for the function $\psi_k$ defined in \eqref{fik}.

\medskip
 To start with, we observe that if $S$ is a compact Riemann surface with non-constant Robin's function then at least two bubbling solutions with $k=1$ exists. Indeed, since $S$ is compact then $H(\xi,\xi)$ attains its minimum and its maximum.  In this case, it is easy to show that $\psi_1$ has two  stable critical point situations. Thus, problem \eqref{mt} has one solution which is bubbling near the global minimizer of $H (\xi , \xi )$, and another solution which is bubbling near the global maximizer, as $\lambda \to 0^+$. Unfortunately, this kind of solutions are hopeless to be found for instance when $S$ is the unit sphere $\mathbb{S}^2$ in $\R^3$ or  when $S$ is the flat two-torus $T$, since, in these examples, the function $H(\xi,\xi)$ is constant. Nevertheless, in these two examples, we can prove the existence of solutions with $k=2$ bubbling points. The case of the flat torus is particularly surprising.

 \medskip
 If $k=2$, the functional $\psi_k$ in \eqref{fik} takes the simplified form
$$
\psi_2(\xi,m)=\,A \sum_{j=1}^2m_j^2+\sum_{j=1}^km_j^2\log m_j^2 - 8\pi m_1 m_2 G(\xi_1 , \xi_2)
$$
where $A= \log 16-2 -8\pi c$, where $c$ is the constant value of $H(\xi , \xi)$ in the case of the sphere and the flat torus.

 \medskip
 Let us start with $S=\mathbb{S}^2$. Since problem \eqref{mt} is invariant under rotations, it is not restrictive to look for solutions with one bubbling point to be a fixed point on $\mathbb{S}^2$, say $\xi_2$. Indeed, by a rotation, one can get another solution bubbling at any other point of $\mathbb{S}^2$, just rotating $\xi_2$ up to this other point.
 Thus, we fix $\xi_2 \in \mathbb{S}^2$ and we are reduced to study the existence of critical points for  $\xi_1 \in \mathbb{S}^2 \mapsto G(\xi_1,\xi_2)$. Notice that $\xi_1 \mapsto G(\xi_1 ,\xi_2)$ has a global minimum.  Then from simple arguments, one sees that $\psi_2$ has a stable critical point situation. We thus get the validity of

\begin{theo}\label{s2}
Assume that $S=\mathbb{S}^2$ is the unit sphere in $\R^3$ and fix $\xi_2 \in \mathbb{S}^2$. Then there exists a family of solutions $u_{\la}$ to problem \eqref{mt} with two  bubbling points such that as $\la\to 0$ the two bubbling points converge to $(\xi_1,\xi_2)$ with $\xi_1$ the global minimum of $G(\cdot,\xi_2)$ and
$${1\over 2} \int_{\mathbb{S}^2} | \nabla u_{\la} |^2_{g_0}\,dv_{g_0} -{\la \over 2}
\int_{\mathbb{S}^2} e^{u^2_{\la}}\,dv_{g_0}=4\pi + o(1), \quad {\mbox {as}} \quad \la \to 0$$
where $g_0$ is the standard round metric on $\mathbb{S}^2$.
\end{theo}
\medskip
  Let us now discuss the case when $S=T$ is a rectangle and we look for solutions $u$ to \eqref{mt} that are  doubly periodic functions on $\fr T$.
  The surprising fact of this case is the multiplicity of solutions.

  \medskip
  Without loss of generality, assume that, in complex notation,
\begin{equation}\label{T}
T=\Big\{z=s {a\over 2}+t{\text{i}b\over 2}\ : \ s,t\in \Big(-{1\over 2},{1\over 2}\Big)\Big\},
\end{equation}
with i the imaginary unity. Since the equation is invariant under translations, if $u$ is a solution to \eqref{mt} then $u(\cdot+p)$ is also a solution to \eqref{mt} for any $p\in T$. In this setting, our next result states the existence of families of solutions of the form \eqref{due}, and satisfying \eqref{uno}, with $k=2$. The exact number of such solutions can be $7$ or $9$, depending on the value of
$$\tau= b/a.
$$
Our result states as follows.

\begin{theo}\label{torusk2}
Assume that $T$ is a rectangle in $\R^2$ given by \eqref{T} and $\tau=b/a$. Then there are $\tau_0<1<\tau_1$ such that if either $\tau\in(0,\tau_0]\cup[\tau_1,+\infty)$ or $\tau\in(\tau_0,\tau_1)$ then there is $\la_0>0$ such that for any $0<\la\le \la_0$ there exist either seven or nine different families of doubly periodic on $\fr T$ bubbling solutions $u_{\la,i}$ respectively to problem \eqref{mt}. These solutions safisfy
\begin{equation}\label{jluik2}
J_\la(u_{\la,i})={1\over 2} \int_T | \nabla u_{\la,i} |^2\,dx -{\la \over 2}
\int_T e^{u^2_{\la,i}}\,dx=4\pi + O(\la),
\end{equation}
as $\la\to 0$. Moreover, there exist bubbling points $\xi_{\la,i}=(\xi_{\la,i}^{(1)},\xi_{\la,i}^{(2)})\in T^2$ and weights \linebreak $m_{\la,i}=(m_{\la,i}^{(1)},m_{\la,i}^{(2)})\in \R^2_+$ such that, passing to a subsequence, $(\xi_{\la,i},m_{\la,i})\to (\xi_{0,i},m_{0,i})$ with $\grad\psi_2(\xi_{0,i},m_{0,i})=0$ and
\begin{equation}\label{abst}
u_{\la,i} (x) \ =\ \sqrt{\la} \, \left (\, 8\pi m_{\la,i}^{(1)}G\big(x,\xi_{\la,i}^{(1)}\big) + 8\pi m_{\la,i}^{(2)}G\big(x,\xi_{\la,i}^{(2)}\big) \ + \ O(\la) \, \right )
\end{equation}
as $\la\to 0$, uniformly on compact subsets of $T\setminus \{\xi_{0,i}^{(1)},\xi_{0,i}^{(2)}\}$, where $\xi_{0,i}=(\xi_{0,i}^{(1)},\xi_{0,i}^{(2)})$.
Here we intend that $i\in\{1,\dots,7\}$ when $\tau\in(0,\tau_0]\cup[\tau_1,+\infty)$, and $i\in\{1,\dots,9\}$ when $\tau\in ( \tau_0,\tau_1)$.

\end{theo}

\medskip

The location of the two bubbling points for the solutions in this result is completely determined. Indeed, fixing $i\in\{1,\dots,7\}$ when $\tau\in(0,\tau_0]\cup[\tau_1,+\infty)$ or $i\in\{1,\dots,9\}$ when $\tau\in ( \tau_0,\tau_1)$ the bubbling points $\xi_{\la,i}^{(1)}$ and $\xi_{\la,i}^{(2)}$ satisfy $\xi_{\la,i}^{(1)}-\xi_{\la,i}^{(2)}\to p_j$ for some $j\in\{1,2,3\}$ as $\la\to 0$ where $p_1$, $p_2$ and $p_3$ are the half periods of $T$:
\begin{equation}\label{halph}
p_1=\dfrac{a}{2} , \quad p_2=\dfrac{\text{i}b}{2}, \quad p_3=\dfrac{a+\text{i}b}{2}.
\end{equation}
\medskip
Moreover, we can show some properties of the weights $m_i$'s. If $0<\tau\le \tau_0$ then
\begin{itemize}
\item  there are two bubbling points $\big(\xi_{\la}^{(1)},\xi_{\la}^{(2)}\big)$ satisfying $\xi_{\la}^{(1)}-\xi_{\la}^{(2)}\to p_1$, and three different pairs of weights $(m_{\la,i}^{(1)},m_{\la,i}^{(2)})$ converging to either $(m_{0},m_{0})$, $(m_{1},m_{2})$ or $(m_2,m_1)$ as $\lambda\to0$ for some $m_0$, $m_1$ and $m_2$ with $m_0\ne m_1$, $m_0\ne m_2$ and $m_1\ne m_2$ such that they give rise to three bubbling solutions;

\item  there are two bubbling points $\big(\xi_{\la}^{(1)},\xi_{\la}^{(2)}\big)$ satisfying $\xi_{\la}^{(1)}-\xi_{\la}^{(2)}\to p_2$, and only a pair of weights $(m_{\la,1},m_{\la,2})$ converging to $(m_{3},m_{3})$ as $\lambda\to0$ for some $m_3$ such that it gives rise to a bubbling solution; and

\item  there are two bubbling points $(\xi_{\la}^{(1)},\xi_{\la}^{(2)})$ satisfying $\xi_{\la}^{(1)}-\xi_{\la}^{(2)}\to p_3$, and three different pairs of weights $(m_{\la,i}^{(1)},m_{\la,i}^{(2)})$ converging to either $(m_{4},m_{4})$, $(m_{5},m_{6})$ or $(m_6,m_5)$ as $\lambda\to0$ for some $m_4$, $m_5$ and $m_6$ with $m_4\ne m_5$, $m_4\ne m_6$ and $m_5\ne m_6$ such that they give rise to three bubbling solutions.
\end{itemize}
If $\tau\ge \tau_1$ then
\begin{itemize}
\item  there are two bubbling points $\big(\xi_{\la}^{(1)},\xi_{\la}^{(2)}\big)$ satisfying $\xi_{\la}^{(1)}-\xi_{\la}^{(2)}\to p_1$, and only a pair of weights $(m_{\la,i}^{(1)},m_{\la,i}^{(2)})$ converging to either $(m_{0},m_{0})$ as $\lambda\to0$ for some $m_0$ such that it gives rise to only a bubbling solution;

\item  there are two bubbling points $\big(\xi_{\la}^{(1)},\xi_{\la}^{(2)}\big)$ satisfying $\xi_{\la}^{(1)}-\xi_{\la}^{(2)}\to p_2$, and three different pairs of weights $(m_{\la,1},m_{\la,2})$ converging to either $(m_1,m_1)$, $(m_2,m_3)$ or $(m_3,m_2)$ as $\la\to 0$ for some $m_1$, $m_2$ and $m_3$ with $m_1\ne m_2$, $m_2\ne m_3$ and $m_1\ne m_3$ such that they give rise to three bubbling solution; and

\item  there are two bubbling points $(\xi_{\la}^{(1)},\xi_{\la}^{(2)})$ satisfying $\xi_{\la}^{(1)}-\xi_{\la}^{(2)}\to p_3$, and three different pairs of weights $(m_{\la,i}^{(1)},m_{\la,i}^{(2)})$ converging to either $(m_{4},m_{4})$, $(m_{5},m_{6})$ or $(m_6,m_5)$ as $\lambda\to0$ for some $m_4$, $m_5$ and $m_6$ with $m_4\ne m_5$, $m_4\ne m_6$ and $m_5\ne m_6$ such that they give rise to three bubbling solutions.
\end{itemize}
If $\tau\in( \tau_0,\tau_1)$ then for every $j=1,2,3$
\begin{itemize}

\item  there are two bubbling points $\big(\xi_{\la,j}^{(1)},\xi_{\la,j}^{(2)}\big)$ satisfying $\xi_{\la,j}^{(1)}-\xi_{\la,j}^{(2)}\to p_j$, and three different pairs of weights $(m_{\la,i,j}^{(1)},m_{\la,i,j}^{(2)})$ converging to either $(m_{0,j},m_{0,j})$, $(m_{1,j},m_{2,j})$ or $(m_{2,j},m_{1,j})$ as $\lambda\to0$ for some differents $m_{0,j}$, $m_{1,j}$ and $m_{2,j}$ such that they give rise to three bubbling solutions.

\end{itemize}

\medskip
 A very interesting situation in the rectangular case is when $a=b$ (or $\tau=1$), namely, in case of a square. Recall that $\tau_0<1<\tau_1$. This fact follows from the analysis in the proof of the previous result, but we highlight it due to the multiplicity of bubbling solutions we obtain: {\it there exist {\bf nine} different families of doubly periodic on $\fr T$ bubbling solutions to problem \eqref{mt} in any square}.

\begin{theo}\label{bubblingsquare}
Assume that $T$ is a square in $\R^2$. Then for any $\la$ small enough and for every half period $p_j$, $j\in\{1,2,3\}$, see \eqref{halph}, there exist bubbling points $\xi_{\la,j}=(\xi_{\la,j}^{(1)},\xi_{\la,j}^{(2)})\in T^2$ and three different pairs of weights $m_{\la,i,j}=(m_{\la,i,j}^{(1)},m_{\la,i,j}^{(2)})\in \R^2_+$, $i=1,2,3$ giving rise to nine bubbling solutions $u_{\la,i,j}$, satisfying \eqref{jluik2} as $\la\to 0$ for $i=1,2,3$ such that, passing to a subsequence, $(\xi_{\la,j},m_{\la,i,j})\to (\xi_{0,j},m_{0,i,j})$, $\xi_{\la,j}^{(1)}-\xi_{\la,j}^{(2)}\to p_j$ and the property \eqref{abst} holds as $\la\to 0$, uniformly on compact subsets of $T\setminus \{\xi_{0,j}^{(1)},\xi_{0,j}^{(2)}\}$, where $\xi_{0,j}=(\xi_{0,j}^{(1)},\xi_{0,j}^{(2)})$.

\end{theo}

\medskip
Theorems \ref{torusk2} and \ref{bubblingsquare} follow from the fact that the existence of nondegenerate critical points of $\psi_2$ is a stable critical point situation.  From similar ideas follows Theorem \ref{s2}, studying first critical points of $\xi_1$'s and then the weights $m_i$'s .

\medskip
For the case $k\geq 3$, or the case $k=2$ on a surface $S$ where the function $H(\xi , \xi) $ is not constant, the analysis of the map
 $(\xi, m) \mapsto \psi_k(\xi,m)$ is much harder.

\medskip
We conclude our introduction mentioning
 the link between the theorems \ref{torusk2} - \ref{s2} and the results contained in \cite{dmr1,dmr2} on concentration phenomena for the Liouville-type problem
\begin{equation}\label{ledbc}
\Delta u+\e^2e^{u}=0,\text{in
$\Om$},\quad
u=0,\text{on }\fr\Om,
\end{equation}
where $\Om$ is bounded smooth domain in $\R^2$, see \cite{bp,DeKM,EGP} and references therein. Our results are also connected to those for Liouville-type equations on compact Riemann surfaces, see \cite{CL0,CL,CLW,DEFM, EsFi,F}. The fine blow-up structure for Liouville-type equations on domains in $\R^2$ or on manifolds very close to the bubbling points is similar to that we found in the problems we are discussing in this paper, nevetheless scalings and intermediate regimes are much more subtle for doubly exponential nonlinearities. Even though  the choice of our first approximation to our bubbling solutions is inspired by the discovery of the blow-up shapes which was obtained first in \cite{adistruwe} and then in \cite{AdDru,Dru}, in our problem, more accurate information is needed, due to the role of the distinct weights $m_j$'s, which were discovered  in \cite{dmr1}. In fact, the presence of the weights $m_j$ marks a strong difference between double exponential nonlinearity and  Liouville type nonlinearity.

\medskip
As in the usual Lyapunov-Schmidt scheme, the strategy of the proof involves linearization about a first approximation, to later reduce the problem to a finite dimensional variational one of adjusting the bubbling centers and the corresponding weights. The critical character of this nonlinearity is very much reflected in the delicate error terms left by the first approximation, which makes the linear elliptic theory needed fairly subtle because of the multiple-regime in the error size and adapted to the Riemann surface $S$ through the use isothermal coordinates.

\medskip
The paper is organized as follows: in Section \ref{approx}, we construct a first approximation to a solution to \eqref{mt} with the required properties and we estimate the size of the error of approximation with appropriate norms. In Section \ref{variat} we describe the scheme of our proofs, by stating the principal results we need, and we give the proof of our Theorems. Section \ref{sec4} is devoted to the computation of the expansion of the energy functional on the first approximation we constructed in Section \ref{approx}. Sections \ref{appeA}, \ref{appeB} and \ref{appeC} are devoted to rigorously prove the intermediate results we state in Section \ref{variat}.

\bigskip
\section{Approximation of the solution}\label{approx}
\noindent It is convenient for our purposes to rewrite problem \eqref{mt} by
replacing $u=\sqrt{\la} v$, so that the problem becomes
\begin{equation}\label{mtv}
\Delta_g v+\la\lf(ve^{\la v^2}-{1\over |S|}\int_S ve^{\la v^2}\,dv_g\rg)=0,
\end{equation}
with $\int_Sv\,dv_g=0$. Following \cite{dmr1}, to construct approximating solutions of \eqref{mtv},
the main idea is to use as ``basic cells'' the functions
\begin{equation*}
u_{\delta,\xi}(x)=w_\de (x-\xi) \qquad \de>0,\: \xi\in\R^2,
\end{equation*} where
\begin{equation}
w_{\mu}(y) := \log \frac{8\mu^2}{(\mu^2 + |y|^2)^2},\quad \mu>0.
\label{wmu}\end{equation}
They are all the solutions of
\begin{equation*}
\left\{ \begin{array}{ll}\Delta u+e^{u}=0 &\text{in $\R^2$}\\
\ds \int_{\R^2} e^u <+\infty, & \end{array} \right.
\end{equation*}
and do satisfy the following concentration property:
$$e^{u_{\delta,\xi}}\rightharpoonup 8\pi\delta_\xi
\quad\text{in measure sense}$$ as $\delta \to 0$. We will use now
isothermal coordinates to pull-back $u_{\delta,\xi}$ in $S$ as in \cite{EsFi}.

\medskip \noindent Let us recall that every Riemann surface $(S,g)$ is locally
conformally flat, and the local coordinates in which $g$ is
conformal to the Euclidean metric are referred to as isothermal
coordinates (see for example the simple existence proof provided
by Chern \cite{Chern}). For every $\xi \in S$ it amounts to find a
local chart $y_\xi$, with $y_\xi(\xi)=0$, from a neighborhood of
$\xi$ onto $B_{2r_0}(0)$ (the choice of $r_0$ is independent of
$\xi$) in which $g=e^{\hat \varphi_\xi(y_\xi(x))}dx$, where $\hat
\varphi_\xi \in C^\infty(B_{2r_0}(0),\mathbb{R})$. In particular,
$\hat \varphi_\xi$ relates with the Gaussian curvature $K$ of
$(S,g)$ through the relation:
\begin{equation} \label{equationvarphi}
\Delta \hat \varphi_\xi(y) =-2K(y_\xi^{-1}(y)) e^{\hat
\varphi_\xi(y)} \qquad \hbox{ for }y \in B_{2r_0}(0).
\end{equation}
We can also assume that $y_\xi$, $\hat \varphi_\xi$ depends
smoothly in $\xi$ and that $\hat \varphi_\xi(0)=0$, $\nabla \hat
\varphi_\xi(0)=0$.

\medskip \noindent  We now pull-back $u_{\delta,0}=w_\de$ in $\xi \in S$, for $\delta>0$, by simply
setting
$$U_{\delta,\xi}(x)=w_{\delta}(y_\xi(x))=\log \frac{8\delta^2}{(\delta^2+|y_\xi(x)|^2)^2}$$ for $x
\in y_\xi^{-1}(B_{2r_0}(0))$. Letting $\chi\in
C_0^\infty(B_{2r_0}(0))$ be a radial cut-off function so that
$0\le\chi\le 1$, $\chi\equiv 1$ in $B_{r_0}(0)$, we introduce the
function $PU_{\de,\xi}$ as the unique solution of
\begin{equation}\label{ePu}
\left\{ \begin{array}{ll}
\ds -\Delta_g PU_{\de,\xi} (x)=\chi_\xi(x)
e^{-\varphi_\xi(x)} e^{U_{\de,\xi}(x)}-\frac{1}{|S|}\int_S
\chi_\xi e^{-\varphi_\xi} e^{U_{\de,\xi}} dv_g &\text{in }S\\
\ds \int_S PU_{\de,\xi} dv_g=0,
\end{array}\right.
\end{equation}
where $\chi_\xi(x)=\chi\big( y_\xi(x)\big)$ and $\varphi_\xi(x)=\hat
\varphi_\xi\big(y_\xi(x)\big)$. Notice that the R.H.S. in (\ref{ePu}) has
zero average and smoothly depends in $x$, and then (\ref{ePu}) is
uniquely solvable by a smooth solution $PU_{\de,\xi}$.

\medskip \noindent Let us recall the transformation law for $\Delta_g$ under
conformal changes: if $\tilde g=e^{\varphi} g$, then
\begin{equation} \label{laplacian} \Delta_{\tilde g}=e^{-\varphi} \Delta_g.\end{equation}
Decompose now the Green function $G(x,\xi)$, $\xi \in S$, as
$$G(x,\xi)=-\frac{1}{2\pi} \chi_\xi(x) \log |y_\xi(x)|+H(x,\xi),$$
and by (\ref{green}) then deduce that
\begin{equation*}
\left\{ \begin{array}{ll}
\ds -\Delta_g H= - \frac{1}{2\pi} \lab
\chi_\xi  \,\log |y_\xi(x)| -\frac{1}{\pi}\langle \grad
\chi_\xi,\grad \log
|y_\xi(x)| \rangle_g-\frac{1}{|S|} &\text{in $S$}\\
\ds \int_S H(\cdot,\xi)\, dv_g=\frac{1}{2\pi} \int_S \chi_\xi \log
|y_\xi(\cdot)| dv_g.&
\end{array} \right.
\end{equation*}
We have used that
$$\Delta_g \log |y_\xi(x)|= e^{-\hat \varphi_\xi(y)}
\Delta \log|y| \Big|_{y=y_\xi(x)}=2\pi \delta_\xi$$ in view of
(\ref{laplacian}).

\medskip \noindent For $r\leq 2r_0$ define
$B_r(\xi)=y_\xi^{-1}(B_r(0))$ and $A_{r_1,r_2}(\xi)=B_{r_1}(\xi) \sm
B_{r_2}(\xi)$, for $r_2<r_1\le 2r_0$. Setting $$\Psi_{\de,\xi}(x)= PU_{\de,\xi}(x)-\chi_\xi [U_{\delta,\xi}-\log(8\de^2)]-8\pi H(x,\xi),$$
we have the following asymptotic expansion of $PU_{\de,\xi}$ as
$\delta \to 0$, as shown in \cite{EsFi}:
\begin{lem}\label{ewfxi}
The function $PU_{\delta,\xi}$ satisfies
\begin{equation}\label{eaxi}
PU_{\delta,\xi}=\chi_\xi \lf[U_{\delta,\xi}-\log(8\delta^2)\rg]+
8\pi H(x,\xi)+O(\delta^2|\log
\delta|)
\end{equation}
uniformly in $S$. In particular, there holds
$$PU_{\delta,\xi}=8\pi G(x,\xi)+O(\delta^2|\log \delta|)$$
locally uniformly in $S \sm\{\xi\}$.
\end{lem}

\noindent The ansatz will be constructed as follows. Given $k \in
\mathbb{N}$, let us consider distinct points $\xi_j \in S$, $m_j>0$ and
$\de_j=\mu_j\e_j>0$, $j=1,\dots,k$. In order to have a good approximation, we will assume that the parameters $\mu_j$'s and $\e_j$'s are given by
\begin{equation}
\log (8\mu_j^2) =
    -2\log (2m_j^2)  + 8\pi H(\xi_j,\xi_j) + 8\pi\sum_{i=1,i\ne j}^k m_im_j^{-1} G(\xi_i,\xi_j),\quad\text{for all }j=1,\dots,k
\label{muj}\end{equation}
and
\begin{equation}\label{e}
\log{1\over \e^4_j}={1\over 2\la m_j^2}-2\log(2m_j^2),\quad\text{for all }j=1,\dots,k.
\end{equation}
Up to take $r_0$ smaller, we
assume that the points $\xi_j$'s are well separated and $m_j's$
are in a compact subset of $(0,+\infty)$, namely, we choose
$\xi=(\xi_1,\dots,\xi_k)\in\Xi$ and $m=(m_1,\dots,m_k)\in \ml{M}$,
where
\begin{equation*}
\Xi=\{(\xi_1,\dots,\xi_k) \in S^k \mid d_g(\xi_i,\xi_j)\geq 4r_0
\:\:\forall\:i,j=1,\dots,k,\:i\not=j\}\quad\text{and}\quad\ml{M}=\Big[\de_0,{1\over
\de_0}\Big]^k,
\end{equation*}
for some small fixed constant $\de_0>0$. Denote $U_j:= U_{\mu_j\e_j,\xi_j}$,
$j=1,\dots,k$. Thus, our approximating solution is
\begin{equation}\label{ansatz}
V(x)= \sum_{j=1}^k m_jPU_j(x),\quad x\in S,
\end{equation}
where $P$ is the projection operator defined by \eqref{ePu}. Notice that $\la\to 0$ if and only if $\e_j\to0$ for each $j=1,\dots,k$. The idea is that the choice of the numbers $\mu_j,\e_j$ makes the error of approximation for $V$ small around each point $\xi_j$. Let us estimate the error which by definition is
\begin{equation}
\label{erre} R\,  = \, \Delta_g V +  \la\lf(Ve^{\la V^2}-{1\over
|S|}\int_SVe^{\la V^2}\,dv_g\rg).
\end{equation}
Setting $w_j(x)=\ds w_{\mu_j}\Big({y_{\xi_j}(x)\over\e_j}\Big)$ for $x\in B_{2r_0}(\xi_j)$, introduce the following $L^\infty$-weighted norm for bounded functions defined
in $S$
\begin{equation}
\label{starnorm} \| h\|_* = \sup_{x \in S } \rho (x)^{-1} |h(x)|,
\end{equation}
where
$$
\rho (x):=  \sum_{j=1}^k  \chi_{ B_{r_0}(\xi_j) }
(x) \rho_j(x) + 1,$$
with
\begin{equation}\label{roj}
\begin{split}
\rho_j(x)=&\,\chi\Big({r_0|y_{\xi_j}(x)|\over \delta\e_j|\log\e_j|^2}\Big)(1+|w_j|+w_j^2)\e_j^{-2} e^{w_j(x)}\\
&+\Big[1-\chi\Big({2r_0|y_{\xi_j}(x)|\over \delta\e_j|\log\e_j|^2}\Big)\Big]\Big[\{1+|\log|y_{\xi_j}(x)|\}e^{\la m_j^2w_j^2(x)}+\la^{-1}\Big]\e_j^{-2} e^{w_j(x)},
\end{split}
\end{equation}
$\de>0$ a large fixed constant and $\chi_A$
is the characteristic function of the set $A$.
Thus, we have proven the following fact.
\begin{lem}\label{estrr0}
Assume \eqref{muj}-\equ{e}. There exists a constant $C>0$,
independent of $\la>0$ small, \st
\begin{equation}\label{re}
\|R\|_*\le  C \la
\end{equation}
for all $\xi \in \Xi$, and $m\in \ml{M}$.
\end{lem}
\begin{proof}[\dem]First, notice that for $x\in B_{2r_0}(\xi_j)$ $\ds U_{\mu\e,\xi}(x)-\log(8\mu^2\e^2)=w_\mu\Big({y_\xi(x)\over\e}\Big)-\log(8\mu^2)+\log{1\over\e^4}.$ By \eqref{muj}-\eqref{e} we find that in $B_{r_0}(\xi_j)$
\begin{equation}\label{ttu}
V(x) =  m_j\lf[ w_j(x)  + {1\over 2\la m_j^2} +
\theta_j(x)\rg]
\end{equation}
where
\begin{equation}\label{te}
\begin{split}
\theta_j(x) =&\:8\pi\Big\langle \grad(H(\cdot,\xi_j)\circ
y_{\xi_j}^{-1})(0)+\sum_{i\ne
j}m_im_j^{-1}\grad(G(\cdot,\xi_i)\circ
y_{\xi_j}^{-1})(0),y_{\xi_j}(x)\Big\rangle\\
&+ O(|y_{\xi_j}(x)|^2) + \sum_{i=1}^k O(\e_i^2|\log\e_i|).
\end{split}
\end{equation}
Hence, we obtain that in
$B_{r_0}(\xi_j)$
\begin{equation}\label{lvlv2}
\la V =  \frac 1{2m_j} + \la m_j( w_j +\theta_j)\qquad\text{and}\qquad
\la V^2 =  w_j+\theta_j+\la m_j^2(w_j+\theta_j)^2+{1\over 4\la m^2_j}.
\end{equation}
Thus, from \eqref{lvlv2} we have that in $B_{r_0}(\xi_j)$
\begin{equation}\label{fv}
\begin{split}
\la Ve^{\la V^2}&={1\over 2m_j}e^{1/(4\la m_j^2)}[1+2\la m_j^2(w_j+\theta_j)]e^{w_j+\theta_j+\la m_j^2(w_j+\theta_j)^2}\\
&=m_j[1+2\la
m_j^2(w_j+\theta_j)]\e_j^{-2}e^{w_j+\theta_j+\la m_j^2(w_j+\theta_j)^2},
\end{split}
\end{equation}
in view of $\ds{1\over 4\la m_j^2}=\log{2m_j^2\over\e_j^2}$.
Furthermore, in $S\sm \cup_{j=1}^kB_{r_0}(\xi_j)$ we have that
\begin{equation}\label{vo}
V(x)=\sum_{j=1}^km_j\lf[8\pi
G(x,\xi_j)+O(\e_j^2|\log\e_j|)\rg]=O(1), \end{equation} so that
$\la Ve^{\la V^2}=O(\la)$ in $S\sm \cup_{j=1}^kB_{r_0}(\xi_j)$.

\medskip
On the other hand, from the definition of $V$ it is readily
checked that
\begin{equation}\label{lav}
\Delta_g V=-\sum_{j=1}^k m_j \chi_j e^{-\varphi_j}\e_j^{-2}e^{w_j}
+ {1\over|S|}\sum_{j=1}^k m_j[8\pi+O(\e_j^2)],
\end{equation}
where $\chi_j=\chi_{\xi_j}$, $\varphi_j=\varphi_{\xi_j}$ for
$j=1,\dots,k$ and in view of $e^{U_j}=\e_j^{-2}e^{w_j}$ and
\begin{equation*}
\int_S \chi_j e^{-\varphi_j} \e_j^{-2}e^{w_j}
dv_g=\int_{B_{r_0}(0)} {8 \mu_j^2\e_j^2\over (\mu_j^2\e_j^2 +
|y|^2 )^2} dy+O(\mu_j^2\e_j^2)=8\pi + O(\e_j^2). \end{equation*}

\medskip
Now, let us estimate the integral term. By using \eqref{vo} we find that
\begin{equation}\label{ivev20}
\la\int_SVe^{\la V^2}=\sum_{j=1}^k\int_{B_{r_0}(\xi_j)}\la Ve^{\la V^2}+ O(\la)
\end{equation}
Now, we write as follows for $\de>0$ large enough and fixed (the same as in the definition of $\rho_j$ in \eqref{roj})
$$\la\int_{B_{r_0}(\xi_j)}Ve^{\la V^2}dv_g=\bigg[\int_{A_{r_0,\de\sqrt{\e_j}}(\xi_j)}+\int_{A_{\de\sqrt{\e_j},\de\e_j|\log\e_j|}(\xi_j)}+\int_{B_{\de\e_j|\log\e_j|}(\xi_j)}\bigg]\la Ve^{\la V^2}dv_g.$$
In $A_{r_0,\de\sqrt{\e_j}}(\xi_j)$, we have that uniformly $V(x)=-4m_j\log|y_{\xi_j}(x)|+O(1),$
in view of the expansion in $S\sm\cup_{j=1}^kB_{\de\sqrt{\e_j}}(\xi_j)$
$$PU_j(x)=-4\chi_j(x)\log|y_{\xi_j}(x)| + 8\pi H(x,\xi_j) + \chi_j(x)\log\Big(1+{\mu_j^2\e_j^2\over |y_{\xi_j}(x)|^2}\Big)+O(\e_j^2|\log\e_j|),$$
and for $i\ne j$ and $x\in A_{r_0,\de\sqrt{\e_j}}(\xi_j)$,
$PU_i(x)=8\pi G(x,\xi_i)+O(\e_i^2|\log\e_i|)=O(1).$
Hence, we find that
\begin{equation*}
\begin{split}
\int_{A_{r_0,\de\sqrt{\e_j}}(\xi_j)} Ve^{\la V^2}\,dv_g&=m_j\int_{A_{r_0,\de\sqrt{\e_j}}(\xi_j)}[-4\log|y_{\xi_j}(x)|+O(1)]e^{\la m_j^2[16\log^2|y_{\xi_j}(x)|+O(|\log|y_{\xi_j}(x)||)]}dv_g\\
&=m_j\int_{B_{r_0}(0)\sm B_{\de\sqrt{\e_j}}(0)}[-4\log|y|+O(1)]e^{\la m_j^2[16\log^2|y|+O(|\log|y||)]}e^{\hat \varphi_j(y)} dy\\
&=O\lf(\int_{B_{r_0}(0)\sm B_{\de\sqrt{\e_j}}(0)}\big|\log|y|\big|\,e^{16\la m_j^2\log^2|y|}dy\rg)=O(1),
\end{split}
\end{equation*}
in view of $y=y_{\xi_j}(x)$, $e^{\hat\varphi_j}=O(1)$, $\la m_j^2|\log|y||=O(1)$ in the considered region and
\begin{equation*}
\begin{split}
\int_{B_{r_0}(0)\sm B_{\de\sqrt{\e_j}}(0)}\big|\log|y|\big|e^{16\la m_j^2\log^2|y|}dy&=2\pi \int_{\de\sqrt{\e_j}}^{r_0} |\log s|e^{16\la m_j^2\log^2 s}s\,ds\\
&=2\pi\int_{\log(\de\sqrt{\e_j})}^{\log r_0}|t|e^{2t+16\la m_j^2t^2}\,dt\quad(t=\log s)\\
&=O\lf(\int_{\log(\de\sqrt{\e_j})}^{\log r_0}|t|e^{t}\,dt\rg)=O(1),
\end{split}
\end{equation*}
since for $r_0<1$ (if necessary), ${1\over2}\log\e_j+\log\de\le t\le\log r_0<0$ and \eqref{e} implies that
$$(16\la m_j^2 t+2)t\le(1+4\la m_j^2\log(2m_j^2)+16\la m_j^2\log\de)t\le t+\alpha\qquad\text{for some constant }\alpha.$$
Now, we get that $\de\e_j|\log\e_j|\le|y_{\xi_j}(x)|\le\de\sqrt{\e_j}$ implies that
$$2\la\log\e_j+\la\log{8\mu_j^2\over (\mu_j^2\e_j^2+\de^2)^2}\le\la w_j(x)\le -4\la\log|\log\e_j|+\la\log{8\mu_j^2\over ({\mu_j^2\over |\log\e_j|^2}+\de^2)^2}<0,$$
for $\la$ small enough, so that, we find that $\la w_j=O(1)$ uniformly in $A_{\de\sqrt{\e_j},\de\e_j|\log\e_j|}(\xi_j)$. Furthermore, it follows that
$$w_j(1+\la m_j^2w_j)\le w_j\lf(1+2\la m_j^2\log\e_j+\la\log{8\mu_j^2\over (\mu_j^2\e_j^2+\de^2)^2}\rg)\le {3\over4}w_j+\beta,$$ for some  constant $\beta$
in $A_{\de\sqrt{\e_j},\de\e_j|\log\e_j|}$, in view of $2\la m_j^2\log\e_j=\la m_j^2\log(2m_j^2)-{1\over 4}$. Hence, by using \eqref{fv}, $\theta_j=O(1)$ and scaling $\e_jz=y_{\xi_j}(x)$, we obtain that
\begin{equation*}
\begin{split}
\int_{A_{\de\sqrt{\e_j},\de\e_j|\log\e_j|}}\la Ve^{\la V^2}&=\,O\bigg(\int_{A_{\de\sqrt{\e_j},\de\e_j|\log\e_j|}}\e_j^{-2}e^{w_j+\la m_j^2w_j^2}\bigg)=O\bigg(\int_{A_{\de\sqrt{\e_j},\de\e_j|\log\e_j|}}\e_j^{-2}e^{{3\over4}w_j}\bigg)\\
&=\,O\bigg(\int_{B_{\de/\sqrt{\e_j}}(0)\sm B_{\de |\log\e_j|}(0)}\exp\Big({3\over4}\log{8\mu_j^2\over (\mu_j^2+|z|^2)^2}\Big)\,dz\bigg)\\
&=\,O\bigg(\int_{\de|\log\e_j|}^{\de/\sqrt{\e_j}}\Big({8\mu_j^2\over(\mu_j^2+s^2)^2}\Big)^{3/4}s\,ds\bigg)=O\bigg(\int_{\de|\log\e_j|}^{\de/\sqrt{\e_j}}{ds\over s^2}\bigg)=O(\la).
\end{split}
\end{equation*}
In the ball $B_{\de \e_j|\log\e_j|}(\xi_j)$, we have that
\begin{equation}\label{ivev2}
\begin{split}
\la\int_{B_{\de \e_j|\log\e_j|}(\xi_j)}&Ve^{\la V^2}=\int_{B_{\de\e_j|\log\e_j|}(\xi_j)}m_j[1+2\la m_j^2(w_j+\theta_j)]\e_j^{-2}e^{w_j+\theta_j+\la m_j^2(w_j+\theta_j)^2}\\
&=m_j\bigg[\int_{B_{\de\e_j|\log\e_j|}(\xi_j)} \e_j^{-2}e^{w_j + \theta_j}(1+O(\la w_j^2+\la |w_j|+\la))\,dv_g\\
&\qquad +2\la m_j^2\int_{B_{\de\e_j|\log\e_j|}(\xi_j)} \e_j^{-2}e^{w_j} O\lf([1+|w_j|][1+\la w_j^2+\la|w_j|]\rg)\,dv_g\bigg] \\
&=m_j[8\pi+O(\la)]
\end{split}
\end{equation}
in view of $\theta_j=O(1)$, $\la(w_j+w_j^2)=O(1)$,
$$\int_{B_{\de\e_j|\log\e_j|}(\xi_j)}\e_j^{-2}e^{w_j+\theta_j}=8\pi +O\bigg(\la+\sum_{j=1}^k\e_j^2|\log\e_j|\bigg)$$
by using \eqref{te} (and similar expansion for $e^{\theta_j}$), scaling $\e_jz=y_{\xi_j}(x)$ so that $\e_j^2 dy=e^{-\varphi_j(x)} dv_g$ and
$$\int_{B_{\de\e_j|\log\e_j|}(\xi_j)}\e_j^{-2}e^{w_j}(|w_j|+w_j^2+|w_j|^3)
dv_g=O(1),$$ since $0\le|y_{\xi_j}(x)|\le\de\e_j|\log\e_j|$ implies that
$$-4\log|\log\e_j|+\log{8\mu_j^2\over(\mu_j^2|\log\e_j|^{-2}+\de^2)^2}\le w_j(x)\le \log{8\over\mu_j^2},$$
namely, $w_j=O(|\log\la|)$ in $B_{\de\e_j|\log\e_j|}(\xi_j)$.
Therefore, we conclude that
\begin{equation}\label{isvev2}
\la \int_S Ve^{\la V^2}\,dv_g=8\pi \sum_{j=1}^km_j+O(\la).
\end{equation}
Notice that from \eqref{e}, we find that
$\e_j^2|\log\e_j|=O(\la)$, for all $j=1,\dots,k$.

\medskip

From \eqref{lav}, \eqref{fv}, \eqref{vo} and \eqref{ivev2} it
follows that
\begin{itemize}
\item in $S \setminus \cup_{j=1}^m B_{r_0}(\xi_j)$ there holds $R=O(\la)$;
\item in $B_{r_0}(\xi_j)$, $j\in\{1,\dots,k\}$, there holds
$$
R = m_j \e_j^{-2} e^{w_j} \left(\lf[1+ 2\la
m_j^2(w_j+\theta_j)\rg] e^{\la
m_j^2(w_j+\theta_j)^2+\theta_j}-e^{-\varphi_j}
 \right ) + O(\la).
$$
\end{itemize}
Observe that for $x\in B_{\e_j r_0}(\xi_j)$ we have that $ R=O(\la
\e_j^{-2} e^{w_j}+\la)$, since $w_j=O(1)$ uniformly in $B_{\e_j
r_0}(\xi_j)$. Moreover, $w_j=O(|\log\la|)$ in $B_{\de\e_j|\log\e_j|^2}(\xi_j)$ and hence,
$$
R =\e_j^{-2} e^{w_j} O\left(\la|w_j|+\la w_j^2+\sum_{j=1}^k\e_j|\log\e_j|^2
\right ) + O(\la)
$$ in view of $\la(w_j+\theta_j)^2=O(\la|\log\la|^2)$ and $\theta_j(x)=O(|y_{\xi_j}(x)|+\sum_{j=1}^k\e_j|\log\e_j|^2)$. Furthermore, from the choice of $\e_j$,
$j=1,\dots,k$ \eqref{e} it follows that
\begin{equation*}
\begin{split}
R =&\, m_j \e_j^{-2} e^{w_j} \Big(\lf[-4\log(\mu_j^2\e_j^2+|y_{\xi_j}(x)|^2)+4\log(2m_j^2)+2\log(8\mu_j^2)+2\theta_j\rg]\\
&\qquad\qquad\qquad\times \la m_j^2 e^{\la
m_j^2 w_j^2+\la m_j^2\theta_j^2+\theta_j(1+2\la m_j^2w_j)}-e^{-\varphi_j}  \Big) + O(\la)\\
=&\,\la O\lf(\lf[\big|\log|y_{\xi_j}(x)|\big|+1\rg]\e_j^{-2}e^{w_j+\la m^2_jw_j^2}+ \la^{-1}\e_j^{-2}e^{w_j}\rg)+ O(\la),
\end{split}
\end{equation*}
in $A_{r_0,\de \e_j|\log\e_j|^2}(\xi_j)$, in view of $\log(\mu_j^2\e_j^2+|y_{\xi_j}(x)|^2)=2\log|y_{\xi_j}(x)|+O(1)$. In particular, $w_j+{1\over 2\la m_j^2}=-4\log|y_{\xi_j}(x)|+O(1)$, so that, in $A_{r_0,\de\sqrt{\e_j}}(\xi_j)$ it holds $\e_j^{-2}e^{w_j+\la m_j^2w_j^2}=O(e^{16\la m_j^2\log^2|y_{\xi_j}(x)|})$, in view of $\la w_j=O(1)$ uniformly
in $B_{2r_0}(\xi_j)$. Hence, the error of approximation satisfies
the global bound
$$
|R(x)| \le    C\la \rho(x).
$$
This completes the proof.
\end{proof}


\medskip
For simplicity, here and in what follows $f$ designates the
nonlinearity
\begin{equation}
\label{effe} f(v) = \la v e^{\la v^2}.
\end{equation}
Now, we will look for a solution $v$ of \eqref{mtv} in the form
$v=V+\phi$, for some small remainder term $\phi$. In terms of
$\phi$, the problem \eqref{mtv} is equivalent to find $\phi\in
\bar H$ so that
\begin{equation}\label{eqphi}
\begin{split}
\Delta_g \phi + f'(V)\phi - {1 \over|S|}\int_S
f'(V)\phi\, dv_g=&\,-R-\bigg[f(V+\phi)-f(V)-f'(V)\phi\\
&\,-{1\over |S|}\int_S
\lf[f(V+\phi)-f(V)-f'(V)\phi\rg]dv_g\bigg]
\end{split}
\end{equation}
Here, it is clear that $f'(V)=\la e^{\la V^2}(1+2\la V^2)$. However, instead of solving directly the problem \eqref{eqphi} we shall study a different problem. To this purpose we need to estimate $f'(V)$. Thus, denoting
\begin{equation}\label{defK}
\ds K:=\sum_{j=1}^k\chi_je^{-\varphi_j}\e_j^{-2}e^{w_j}
\end{equation}
we have the following result.

\begin{lem}
Assume \eqref{muj}-\equ{e}. There exists a constant $C>0$,
independent of $\la>0$ small, \st
\begin{equation}\label{k}
\|f'(V)-K\|_*\le  C \la
\end{equation}
for all $\xi \in \Xi$, and $m\in \ml{M}$.
\end{lem}

\begin{proof}[\dem]
From \eqref{lvlv2} it follows  that in $B_{r_0}(\xi_j)$
$$e^{\la V^2}=2m_j^2\e_j^{-2}e^{w_j+\theta_j+\la m_j^2(w_j+\theta_j)^2}\quad\text{and}\quad 2\la V^2=2(w_j+\theta_j)+2\la m_j^2(w_j+\theta_j)^2+{1\over 2\la m_j^2},$$
so that, in $B_{r_0}(\xi_j)$
$$f'(V)=(1+2\la m_j^2+4\la m_j^2(w_j+\theta_j)+4\la^2 m_j^4(w_j+\theta_j)^2)\e_j^{-2}e^{w_j+\theta_j+\la m_j^2(w_j+\theta_j)^2}.$$
Thus, it is clear that $f'(V)=\e_j^{-2}e^{w_j}(1+O(\la))$ uniformly in $B_{\de\e_j}(\xi_j)$ and $f'(V)=O(\la)$ in $S\sm\cup_{j=1}^kB_{r_0}(\xi_j)$. Furthermore, we have that
\begin{equation*}
\begin{split}
f'(V)-K=\e_j^{-2}e^{w_j}&\Big[\lf(1+2\la m_j^2+4\la m_j^2(w_j+\theta_j)+4\la^2 m_j^4(w_j+\theta_j)^2\rg)\\
&\quad\times\e_j^{-2}e^{w_j+\theta_j+\la m_j^2(w_j+\theta_j)^2}-e^{-\varphi_j}\Big]
 \end{split}
\end{equation*}
uniformly in $B_{r_0}(\xi_j)$ and $\ds f'(V)-K =O(\la)$ in $S\sm\cup_{j=1}^kB_{r_0}(\xi_j)$.  Similar to the estimate \eqref{re}, we conclude \eqref{k}.
\end{proof}

\medskip

In order to simplify the arguments, in view of \eqref{k}, we write \eqref{eqphi} in the form
\begin{equation}\label{ephi}
L(\phi)=-[R+N(\phi)] \qquad\text{ in $S$},
\end{equation}
where the linear operator $L$ is defined as
\begin{equation}\label{ol}
L(\phi) = \Delta_g \phi + K\phi - {1 \over|S|}\int_S
K\phi\, dv_g,
\end{equation}
and the nonlinear part $N$ is given by
\begin{equation}\label{nlt}
\begin{split}
N(\phi)=&\,f(V+\phi)-f(V)-f'(V)\phi-{1\over |S|}\int_S
\lf[f(V+\phi)-f(V)-f'(V)\phi\rg]\\
&\,+[f'(V)-K]\phi-{1\over |S|}\int_S[f'(V)-K]\phi\,dv_g.
\end{split}
\end{equation}
Notice that for
all $\phi \in \bar H$
$$\int_S L(\phi) dv_g=\int_S N(\phi)dv_g=\int_S R dv_g=0.$$


\section{Variational reduction and proof of main results}\label{variat}
\noindent In the so-called nonlinear Lyapunov-Schimdt reduction,
an important step is the solvability theory for the linear operator, obtained as the linearization of
\eqref{mtv} at the approximating solution $V$, namely, \eqref{eqphi}. In our approach, in order to simplify the arguments we will study the operator $L$ given in (\ref{ol}) under suitable orthogonality conditions. Let us observe that $L(\phi)=\ti L(\phi) + c(\phi)$, for functions $\phi$ defined on $S$, with
\begin{equation}\label{dolti}
\ti L(\phi)=\Delta_g\phi + K\phi
\end{equation} and
$\ds c(\phi):=-\frac{1}{|S|}\int_{S} K\phi\, dv_g$, where $K$ is given by \eqref{defK}.  Observe that, as $\la \to 0$, formally the operator $\ti L$, scaled and centered at $0$
by setting $y =y_{\xi_j}(x)/\e_j $ for $x\in B_{\e_jr_0}(\xi_j)$, approaches $\hat L_j$
defined in $\mathbb{R}^2$ as
$$\hat L_j(\phi) = \Delta\phi+{8\mu_j^2\over (\mu_j^2+|y|^2)^2}\phi.$$ Due to
the intrinsic invariances, the kernel of $\hat L_j$ in
$L^\infty(\R^2)$  is non-empty and is spanned by $Y_{ij}$,
$i=0,1,2$, where
\begin{equation*}
Y_{ij}(y) = { 4 \mu_jy_i \over \mu_j^2+|y|^2} ,\qquad
i=1,2,\qquad\text{and}\qquad Y_{0j}(y) = 2\,{\mu_j^2-|y|^2\over \mu_j^2+|y|^2}.
\end{equation*}
Since \cite{DeKM,EGP,dmr1} it is by now rather standard to show the
invertibility of $L$ in a suitable ``orthogonal" space, and a
sketched proof of it will be given in Appendix A. See also \cite{EsFi} for an extension to a Riemann surface. Furthermore, an important goal in the study of the $L$ operator is to get rid of
the presence of the term $c(\phi)$.

\medskip \noindent To be more precise, for $i=0,1,2$ and $j=1,\dots,k$ introduce the functions
\begin{equation*}
Z_{ij}(x) = Y_{ij}\lf({y_{\xi_j}(x)\over
\e_j}\rg)=\left\{\begin{array}{ll}
\ds 2  {\mu_j^2\e_j^2- |y_{\xi_j}(x)|^2\over \mu_j^2\e_j^2+|y_{\xi_j}(x)|^2} &\hbox{for }i=0\\
&\\[-0.2cm]
 \ds {4\mu_j\e_j(y_{\xi_j}(x))_i \over \mu_j^2\e_j^2+|y_{\xi_j}(x)|^2}&\hbox{for }i=1,2. \end{array} \right.
\end{equation*}
and let $PZ_{ij}$ be the projections of $Z_{ij}$ as the solutions in $\bar H$ of
\begin{equation}\label{Pzij}
\left\{ \begin{array}{ll} \ds\Delta_g PZ_{ij} =\chi_j \Delta_g
Z_{ij}-\frac{1}{|S|}\int_S
\chi_j\Delta_g Z_{ij} dv_g &\text{in }S\\
\ds\int_S PZ_{ij} dv_g=0,
\end{array}\right..
\end{equation}
Notice that $- \lap_gZ_{ij}=e^{-\varphi_j}\e_j^{-2}e^{w_j}Z_{ij}$ in $B_{2r_0}(\xi_j)$ for all $i=0,1,2$ and $j=1,\dots,k$. In Appendix A we will prove the following result:
\begin{prop} \label{p2}
There exists $\la_0>0$ so that for any points $\xi=(\xi_1,\dots,\xi_k) \in \Xi$ and $m=(m_1,\dots,m_k)\in\ml{M}$, there is a unique solution $\phi \in \bar H(S) \cap W^{2,2}(S)$ and coefficients $c_{ij}
\in \mathbb{R}$ of
\begin{equation}\label{plco}
\left\{ \begin{array}{ll}
L(\phi) = h + \displaystyle \sum_{i=0}^{2} \sum_{j=1}^m c_{ij} \Delta_g PZ_{ij}&\text{in }S\\
\ds \int_S \phi \Delta_g PZ_{ij} dv_g=0
&\forall\: i=0,1,2,\, j=1,\dots,m
\end{array} \right.
\end{equation}
for all $0<\la<\la_0$, $h\in C(S)$ with $\|h\|_*<+\infty$ and $\ds \int_Sh\,dv_g=0$.
Moreover, the map $(\xi,m) \mapsto (\phi,c_{ij})$ is
differentiable in $(\xi,m)$ with
\begin{eqnarray}
&&\|\phi \|_\infty \le C \|h\|_*\:,\qquad \displaystyle \sum_{i=0}^{2} \sum_{j=1}^k  |c_{ij}|\le C\|h\|_* \label{estmfe1} \\
&& \sum_{j=1}^k \lf(\sum_{i=1}^2
\e_j\|\fr_{(\xi_j)_i} \phi\|_\infty + {1\over |\log \e_j|}
\|\fr_{m_j} \phi \|_\infty\rg) \le C
\|h\|_*\label{estd}
\end{eqnarray}
for some $C>0$.
\end{prop}

\noindent Let us stress that the right hand side of the equation \eqref{plco} of $L(\phi)$ integrates zero.

\medskip \noindent Let us recall that  $v=V+\phi$  solves \eqref{mtv} if $\phi\in \bar H$ does satisfy (\ref{ephi}). Since the operator $L$ is not fully invertible, in view of Proposition \ref{p2} one can solve the nonlinear problem (\ref{ephi}) just up to a linear combination of $\Delta_g PZ_{ij}$'s, as explained in the following (see Appendix B for a proof):
\begin{prop}\label{lpnlabis}
Let $\delta_0,r_0 >0$ small and fixed. Then there exist $\la_0>0$, $C>0$ \st for
$0<\la<\la_0$, for any $\xi=(\xi_1,\dots,\xi_k)\in \Xi$ and $m=(m_1,\dots,m_k)\in\ml{M}$, problem
\begin{equation}\label{pnlabis}
\left\{ \begin{array}{ll} L(\phi)= -[R+N(\phi)] +\displaystyle \sum_{i=0}^{2}\sum_{j=1}^k c_{ij} \Delta_g
PZ_{ij}& \text{in } S\\
\ds \int_S \phi \Delta_g PZ_{ij} dv_g= 0\
\text{ for all } \: i=0,1,2,\, j=1,\dots,k
\end{array} \right.
\end{equation}
admits a unique solution $\phi(\xi,m)\in \bar H\cap W^{2,2}(S)$ and $c_{ij}=c_{ij}(\xi,m)$, $i=0,1,2$, $j=1,\dots,k$ \st
\begin{equation}\label{cotaphi}
\|\phi\|_\infty\le C\la,\qquad \displaystyle \sum_{i=0}^{2} \sum_{j=1}^k  |c_{ij}|\le C\la
\end{equation}
where $R$, $N$ are given by \eqref{erre} and \eqref{nlt}, respectively. Furthermore, the map $(\xi,m)\mapsto\phi(\xi,m)\in C(S)$ is $C^1$ and for $l=1,\dots,k$ we have
\begin{equation}\label{cotadphi}
\|\fr_{\xi_l}\phi\|_\infty\le {C\la\over \e_l} \quad\text{and}\quad\|\fr_{m_l}\phi\|_\infty\le C.
\end{equation}
\end{prop}

\noindent The function $\phi(\xi,m)$ obtained in Proposition \ref{lpnlabis} will be a true solution of \eqref{ephi} if $\xi$ and $m$ are such that $c_{ij}(\xi,m )=0$ for all $i=0,1,2,$ and $j=1,\dots,k$. This problem is equivalent to finding critical
points of the reduced energy
\begin{equation}\label{efela}
\ml{F}_\lambda(\xi, m )=J_\la\big(U(\xi,m)+\ti\phi(\xi,m) \big),
\end{equation}
where $J_\lambda$ is given by \eqref{energy}, $U(\xi,m)=\sqrt{\la}\, V(\xi,m)$ and $\ti\phi(\xi,m)=\sqrt{\la}\,\phi(\xi,m ) $, as stated in (See Appendix C)

\begin{lem}\label{cpfc0bis}
There exists $\la_0$ such that, if $(\xi,m)\in
\Xi \times \ml{M}$ is a critical point of $\ml{F}_\lambda$ with $0<\la<\la_0$, then
$v=V(\xi,m)+\phi(\xi,m)$ is a solution of \eqref{mtv}, i.e., $c_{ij}(\xi,m )=0$ for all $i=0,1,2,$ and
$j=1,\dots,k$.
\end{lem}

\noindent Once equation \eqref{mtv} has been reduced to the
search of c.p.'s for $\ml{F}_\lambda$, it becomes crucial to show that
the main asymptotic term of $\ml{F}_\lambda$ is given by
$J_\lambda ( U )$, for which we also have an expansion. More precisely, in section 5 we will prove
\begin{prop} \label{fullexpansionenergy}
Assume \eqref{muj}-\eqref{e}. The following expansion does hold
\begin{eqnarray} \label{fullJUt}
\ml{F}_\lambda (\xi , m) &=&2\pi k-{\la|S|\over 2}+8\pi\la\psi_k(\xi,m) + \theta_\lambda(\xi, m ) \nonumber
\end{eqnarray}
in $C^1(\Xi\times \ml{M})$ as $\la\to 0^+$, where $\psi_k(\xi,m)=\psi_k(\xi_1,\dots,\xi_k,m_1,\dots,m_k)$ is given by
\eqref{fik}
 and the term $\theta_\lambda(\xi, m )$ satisfies
\begin{eqnarray} \label{rlambda}
|\theta_\la(\xi,m)| +
\sum_{l=1}^k\lf[\sum_{i=1}^2\lf|\fr_{(\xi_l)_i}\theta_\la(\xi,m)\rg| +  \lf|\fr_{m_l}\theta_\la(\xi,m)\rg|\rg]=O(\la^2|\log\la|)
\end{eqnarray}
uniformly for points $(\xi, m )\in \Xi\times\ml{M}$.
\end{prop}


\bigskip

\noindent We are now in position to prove the main results
stated in the Introduction.
\begin{proof}[{\bf Proof (of Theorem \ref{teo2}):}] According to Lemma \ref{cpfc0bis}, we have a solution of problem \eqref{mt} if we adjust $(\xi,m)$ so that it is a critical point of $\ml{F}_\la$ defined by \eqref{efela}. This is equivalent to finding a critical point of
$$\ti{\ml{F}}_\la(\xi,m)={1\over 8\pi\la} \lf[\ml{F}_\la(\xi,m)-2\pi k+{\la |S| \over 2} \rg].$$
Thanks to Proposition \ref{fullexpansionenergy}, we have that the function $\ti{{\mathcal F}}_\la (\xi
, m )$ is $C^1$-close to $\psi_k (\xi , m)$  in $\Xi\times \ml{M}$ when $\la $ is
small enough. Now, let $\DD$ be the open set such that
$$
\bar \DD \subset \{ (\xi , m ) \in S^k \times \R_+^k \, : \,
\xi_i \not= \xi_j , \forall i\not= j \},
$$
where $\psi_k$ has a stable critical point situation. Then any $C^1$-perturbation of $\psi_k$ has a critical point in $\DD$. Thus, choosing $r_0$ and $\de_0$ smaller if necessary so that $\bar\DD\subset \Xi\times\ml{M}$, we conclude that $\ti{\ml{F}}_\la$ has a critical point $(\xi_\la  , m_\la )$
in $\DD$, for all such small $\la$. 
Therefore
$$u_\la (x)  =  U(\xi_\la ,
m_\la) (x) + \ti\phi(\xi_\la  ,  m_\la ) (x)
=\sqrt{\la }\lf[ V(\xi_\la ,
m_\la) (x) + \phi(\xi_\la  ,  m_\la ) (x) \rg]$$
is a solution to our problem \equ{mt}. The qualitative properties of this
solution predicted by Theorem \ref{teo2} are direct consequence of
our construction. This concludes the proof.
\end{proof}

\medskip

\begin{proof}[{\bf Proof (of Theorem \ref{torusk2}):}] We shall apply the result of Theorem \ref{teo2} for the case $k=2$ with $S=T$ a flat two-torus in rectangular form given by \eqref{T}. In this case, it holds that the function $H(\xi,\xi)$ is constant. Notice that on $T$ we have invariance under translations, in other words, if $u$ is a solution to \eqref{mt} then $u(\cdot+p)$ is also a solution to \eqref{mt} for any $p\in T$. Furthermore, by the same property, it is know that the Green's function satisfies $G(\xi_1,\xi_2)=G(\xi_1-\xi_2,0)$. Hence, with a slightly abuse of notation we denote $G(z)$ by the Green's function $G(\cdot, 0)$ and we make the change of variables $z=\xi_1-\xi_2$. Thus, we are reduced to look for critical points of the functional
$$\ml{F}_\lambda (z , m) =4\pi -{\la|T|\over 2}+8\pi\la\psi_2(z,m) + \theta_\lambda(z, m ),$$
with
\begin{equation}\label{fi2}
\psi_2(z,m)=\,A\sum_{j=1}^2 m_j^2+2\sum_{j=1}^2 m_j^2\log m_j-8\pi m_1 m_2G(z).
\end{equation}
where $A$ is an absolute constant $A=\log 16-2-8\pi H(\xi ,\xi)$, and, it is sufficient first to find nondegenerate critical points $z$ of $\psi_2$ (for any $m=(m_1,m_2)$) and hence, to look for nondegenerate critical points $m$ of $\psi_2(z,\cdot)$ (for the latter $z$), so that $(z,m)$ are nondegenerate critical points of $\psi_2$, since this is a stable critical point situation. Therefore, there exist $\la_0>0$ such that for all $0<\la<\la_0$ there exist critical points $(\xi_\la,m_\la)$ of $\ml{F}_\la$. Let us stress that we can find critical points $z$ of $\psi_2$ independent of $m$, in view of \eqref{fi2}.


\begin{claim}
If $T$ is a rectangle then there exist exactly three nondegenerate critical points of $\psi_2(\cdot,m)$ for any $m=(m_1,m_2)$. They are the half periods of $T$: $p_1=\dfrac{a}{2}$, $p_2=\dfrac{\text{\em i}b}{2}$ (saddle points) and $p_3=\dfrac{a+\text{\em i}b}{2}$ (minimum point), with \em{i} the imaginary unity.
\end{claim}

\begin{proof}[\dem]  It is know that the Green's function $G$ has exactly three nondegenerate critical points $p_1$, $p_2$ and $p_3$, which are the half periods of $T$ given by \eqref{halph}, see \cite{CLW}. Hence, choosing $z=p_i$ for some $i=1,2,3$ we have nondegenerate critical points of $\psi_2(\cdot,m)$ in points $z$ for any $m$.
\end{proof}

\medskip

Now, let us look for critical points in $m=(m_1,m_2)$. Thus, we get
$$\fr_{m_i}\psi_2(z,m)=2(A+1)m_i+4m_i \log m_i-8\pi m_j G(z),\qquad i, j =1,2,\ \ j\ne i$$
and we have to find points $m_1,m_2\in(0,+\infty)$ solutions to the system
\begin{equation}\label{smk2}
\lf\{\begin{array}{rcl}
(A+1)m_1+2m_1 \log m_1 & = &4\pi m_2 G(z) \\[0.1cm]
(A+1)m_2+2m_2 \log m_2 & = &4\pi m_1G(z)
\end{array}\rg. .
\end{equation}
Let us stress that for each $z=p_i$, $i=1,2,3$ we look for $m$ a critical point of $\psi_2(p_i,\cdot)$.

\begin{claim}\label{essmk2}
If $G(z)\ge0$ then there exists an only solution $(m_1,m_2)$ to the system \eqref{smk2} and it satifies $m_1=m_2$. If $G(z) < 0$ then there exist exactly three different pairs of solutions $(m_1,m_2)$ to the system \eqref{smk2} in the form $(m_1,m_2)$ or $(m_0,m_0)$ or $(m_2,m_1)$ with some positive real numbers satisfying $m_1<m_0<m_2$.
\end{claim}

\begin{proof}[\dem]
To this aim, denote $B=4\pi G(z)$ and first assume that $B\ne 0$. Consider the function
\begin{equation}\label{f0}
\ds f_0(t)={A+1\over B} t+{2\over B} t\log t
\end{equation}
so that we re-write \eqref{smk2} in the form
\begin{equation}\label{smk2b}
\lf\{\begin{array}{rcl}
f_0( t) & = & s \\[0.1cm]
f_0( s) & = & t
\end{array}\rg..
\end{equation}
Thus, we look for the intersection points between the two curves $s=f_0( t)$ and $t=f_0( s)$ in the plane $ts$. Note that $f_0$ satisfies $f_0(0)=0$, its derivative is $\ds f'_0(t)={A+1\over B} +{2\over B} \log t+{2\over B}$ and hence, $f'_0$ is strictly increasing if $B>0$ and strictly decreasing if $B<0$, so that $f_0$ is strictly convex if $B>0$ and strictly concave in $[0,+\infty)$ if $B>0$ and
$$f'_0(t)\to \begin{cases}
-\infty,&B>0\\
+\infty, &B<0\end{cases}\quad\text{ as $t\to 0^+\ $ and } \ f'_0(t)\to \begin{cases}
+\infty,&B>0\\
-\infty, &B<0\end{cases}\quad\text{ as $t\to +\infty$}.$$

Now, assume that $B> 0 $, namely, $G(z)>0$. From the previous analysis, $f_0$ satisfies that $f(0)=0$, $f_0$ is strictly convex in $[0,+\infty)$, $f'_0(t)\to -\infty$ as $t\to 0^+$ and $f'_0(t)\to +\infty$ as $t\to+\infty$. Furthermore, it is readily check that $f_0$ is strictly decreasing in $[0, e^{-(A+3)/2}]$, its graph is below the axis $t$ in $[0, e^{-(A+3)/2}]$ ($f_0$ is negative), it is strictly increasing in $[e^{-(A+3)/2},+\infty)$ and it has a minimum at $t=e^{-(A+3)/2}$. Therefore, there is a unique $m_0=m_0(B)>0$ such that $f_0(m_0)=m_0$, namely, $m_0$ is the only fixed point of $f_0$, it satisfies
$$(A+1)m_0+2m_0\log m_0=Bm_0$$
and $(m_0,m_0)$ is the only solution to system \eqref{smk2}, since reflecting the curve $s=f_0(t)$ with respect to the line $s=t$ we obtain the curve $t=f_0(s)$ and $(m_0,m_0)$ is the only point of intersection of the curves.

On the other hand, assume that $B<0$, namely, $G(z)<0$. Then, $f_0$ satisfies that $f(0)=0$, $f_0$ is strictly increasing in $[0, e^{-(A+3)/2}]$, its graph is above the diagonal $s=t$ in $[0, e^{-(A+3)/2}]$, it is strictly decreasing in $[e^{-(A+3)/2},+\infty)$, it has a maximum at $t=e^{-(A+3)/2}$, it is strictly concave in $(0,+\infty)$, $f_0(t)\to -\infty$ as $t\to+\infty$, $f'_0(t)\to +\infty$ as $t\to 0^+$ and $f'_0(t)\to -\infty$ as $t\to+\infty$. Notice that reflecting the curve $s=f_0(t)$ with respect to the line $s=t$ we obtain the curve $t=f_0(s)$. Hence, we deduce that there are three distinct points of intersection between the curves $s=f_0( t)$ and $t=f_0( s)$. Therefore, there is a unique $m_0=m_0(B)>0$ such that $f_0(m_0)=m_0$ and two more distinct solutions  $(m_1,m_2)$ and $(m_2,m_1)$ to the system \eqref{smk2} with $m_1<m_0<m_2$ and satisfying $f_0(f_0(m_1))=m_1$ and $f_0(f_0(m_2))=m_2$, namely, with $m_1$ and $m_2$ are fixed points of $f_0\circ f_0$ (but $f_0(m_1)\ne m_1$ and $f_0(m_2)\ne m_2$). Also, it follows that $m_1<e^{-(A+3)/2}<m_0$. Let us stress that $m_0$ is a fixed point of $f_0$ and $m_1$ and $m_2$ are fixed points of $f_0\circ f_0$ but not of $f_0$.

Now, for the case $B=0$, system \eqref{smk2} is reduced to solve the equation
$$ (A+1)t+2t\log t=0\iff t\lf[ A+1 +2\log t\rg]=0.$$
Since we look for solutions $m_1,m_2> 0$, we get that $m_1=m_2=e^{-\frac{A+1}{2}}$ is the only solution to the system \eqref{smk2} with $G(z)=0$.

The proof of the claim is finished.
\end{proof}

\medskip
 To conclude that a critical point $(m_1,m_2)$ of $\psi_2(z,\cdot)$, namely, a solution to the system \eqref{smk2} is nondegenerate we study its Hessian matrix. Notice that if $G(z) =-\dfrac{1}{4\pi}$ then the system \eqref{smk2} has three pairs of solutions in the form $(m_1,m_2)$ or $(m_0,m_0)$ or $(m_2,m_1)$ with some positive real numbers $m_1<m_0<m_2$.

\begin{claim}
If $G(z)\ne -\dfrac{1}{4\pi}$ then all the solutions $(m_1,m_2)$ to the system \eqref{smk2} are nondegenerate critical points of $\psi_2(z,\cdot)$. If $G(z)= -\dfrac{1}{4\pi}$ then the pairs of solutions $(m_1,m_2)$ and $(m_2,m_1)$ to the system \eqref{smk2} are nondegenerate critical points of $\psi_2(z,\cdot)$ and $(m_0,m_0)$ is degenerate.
\end{claim}

\begin{proof}[\dem]
Direct computations lead us to find the determinant of the Hessian matrix as
\begin{equation*}
\begin{split}
\lf|H_m\psi_2(z,m)\rg| 
=& \,\Big(8\pi \dfrac{m_2}{m_1} G(z)+4\Big)\Big(8\pi \dfrac{m_1}{m_2}G(z)+4\Big)-64[\pi G(z)]^2\\[0.2cm]
=&\, 32\pi \Big( \dfrac{m_2}{m_1} + \dfrac{m_1}{m_2}\Big) G(z)+16
\end{split}
\end{equation*}
by using that $(m_1,m_2)$ satisfies \eqref{smk2}. Therefore, if $2\pi \Big( \dfrac{m_2}{m_1} + \dfrac{m_1}{m_2}\Big) G(z)\ne -1$ then $m=(m_1,m_2)$ is a nondegenerate critical point of $\psi_2(z,\cdot)$ (for fixed $z$). In particular, if $G(z)\ge 0$ then we have an only nondegenerate critical point $m=(m_0,m_0)$, with $m_0$ the only fixed point of $f_0$ defined by \eqref{f0}, when $B>0$ or $m_0=e^{-{A+1\over 2}}$ when $B=0$, since $\lf|H_m\psi_2(z,m)\rg|\ge16$. In case $G(z)<0$ we have at most three nondegenerate critical points. More precisely, we have several cases readily checked:
\begin{itemize}
\item $m=(m_0,m_0)$, with $m_0$ the only fixed point of $f_0$ defined by \eqref{f0}, is a nondegenerate critical point of $\psi_2(z,\cdot)$ only if $G(z)\ne -\dfrac{1}{4\pi}$ in view of $\lf|H_m\psi_2(z,m)\rg|=64\pi G(z)+16\ne 0$; in other words, if $G(z)< -\dfrac{1}{4\pi}$ then $\lf|H_m\psi_2(z,m)\rg|=64\pi G(z)+16< 0$ and if $ -\dfrac{1}{4\pi}<G(z)<0$ then $\lf|H_m\psi_2(z,m)\rg|=64\pi G(z)+16> 0$;

\item if $G(z)=-\dfrac{1}{4\pi}$ then on one hand, $(m_0,m_0)$, with $m_0$ the only fixed point of $f_0$ defined by \eqref{f0}, is a degenerate critical point of $\psi_2(z,\cdot)$, in view of $\lf|H_m\psi_2(z,m)\rg|=64\pi G(z)+16= 0$, and on the other hand, both $(m_1,m_2)$ and $(m_2,m_1)$, where $m_1$ and $m_2$ are the fixed points of $f_0\circ f_0$ different than $m_0$ the fixed point of $f_0$, with $m_1<m_0<m_2$, are nondegenerate critical points of $\psi_2(z,\cdot)$ in view of $m_1\ne m_2$, so that, $$2\pi \Big( \dfrac{m_2}{m_1} + \dfrac{m_1}{m_2}\Big) G(z)=-\dfrac{1}{2} \Big( \dfrac{m_2}{m_1} + \dfrac{m_1}{m_2}\Big) \ne -1;\quad\text{ and}$$

\item either $(m_1,m_2)$ or $(m_2,m_1)$, with $m_1$ and $m_2$ the fixed points of $f_0\circ f_0$ different than $m_0$ the fixed point of $f_0$ defined by \eqref{f0}, with $m_1<m_0<m_2$, are nondegenerate critical points of $\psi_2(z,\cdot)$ if $G(z) < -\dfrac{1}{4\pi}$ in view of $\lf|H_m\psi_2(z,m)\rg|\le 64\pi G(z)+16<0$, since $\dfrac{m_2}{m_1} + \dfrac{m_1}{m_2}\ge 2$.

\end{itemize}

Notice that it remains to analyze the critical points $(m_1,m_2)$ or $(m_2,m_1)$, with $m_1$ and $m_2$ the fixed points of $f_0\circ f_0$ different than $m_0$ the fixed point of $f_0$ defined by \eqref{f0} when $-\dfrac{1}{4\pi }<G(z)<0$.  To this aim, we shall use the equation that must satisfy degenerate critical point of $\psi_2(z,\cdot)$. Thus, we get that
$$2\pi \Big( \dfrac{m_2}{m_1} + \dfrac{m_1}{m_2}\Big) G(z)+1=0 \iff 4\pi G(z) m_1^2+2m_1m_2+4\pi G(z)m_2^2=0.$$
Recall that $B=4\pi G(z)$, so that, $-1<B<0$ when $-\dfrac{1}{4\pi }<G(z)<0$. So, we re-write the latter equality in the form $Bt^2+2ts+Bs^2 = 0$ and we consider the system
\begin{equation}\label{smk2c}
\lf\{\begin{array}{rcl}
(A+1) t+2 t\log t & = & Bs \\[0.1cm]
(A+1) s+2 s\log s & = & Bt \\[0.1cm]
Bt^2+2ts+Bs^2& = &0
\end{array}\rg..
\end{equation}
Let us show that this system does not have any solution, so that the critical points $(m_1,m_2)$ or $(m_2,m_1)$ are nondegenerate. Assume that $(t,s)$ is a solution of the system \eqref{smk2c}. From the third equation it follows that 
either $Bt+\lf(1 + \sqrt{1-B^2}\rg)s=0$ or $Bt+\lf(1 - \sqrt{1-B^2}\rg)s=0$. First, assume that $Bt+\lf(1 - \sqrt{1-B^2}\rg)s=0$. Hence, we have that
$$Bt=-\lf(1 - \sqrt{1-B^2}\rg)s\iff Bs=-\lf(1 + \sqrt{1-B^2}\rg)t.$$
So, replacing $Bt$ in the second equation of the system \eqref{smk2c} and since $s\ne 0$, we get that
 $$s =\exp\lf(-\Big[A+2-\sqrt{1-B^2}\Big]/2\rg)\quad\text{ and } \quad t ={-1+\sqrt{1-B^2} \over B} \exp\lf(-\Big[A+2-\sqrt{1-B^2}\Big]/2\rg).$$
On the other hand, similarly as above replacing $Bs$ in the first equation of the system \eqref{smk2c} we get that
$$t=\exp\lf(-\Big[A+2+\sqrt{1-B^2}\Big]/2\rg)\quad\text{and}\quad
s ={-1-\sqrt{1-B^2} \over B} \exp\lf(-\Big[A+2+\sqrt{1-B^2}\Big]/2\rg).$$
If $(t,s)$ is a solution of the system \eqref{smk2c} then necessarily
$$s=\exp\lf(-\Big[A+2-\sqrt{1-B^2}\Big]/2\rg)  ={-1- \sqrt{1-B^2} \over B} \exp\lf(-\Big[A+2+\sqrt{1-B^2}\Big]/2\rg).$$
From this equality we obtain that
\begin{align}\label{eqB}
B\exp\lf(\sqrt{1-B^2} \rg)+ \sqrt{1-B^2} + 1&=0.
\end{align}
Performing the change of variable $t=\sqrt{1-B^2}$, we have that $0<t<1$, $B=-\sqrt{1-t^2}$ and $t$ satisfies the equation
$$-\sqrt{1-t^2} e^t+t+1=0.$$
It turns out that the function $g(t)=\sqrt{1-t^2} e^t$ satisfies $g(0)=1$, $g'(0)=1$,
 $g$ is strictly concave and its tangent line at $t=0$ is exactly $s=t+1$. Therefore, $g(t)<t+1$ for all $0<t<1$ and the equality is attained at $t=0$. In other words, there is no $B$ with $-1<B<0$ satisfying the equation \eqref{eqB}. In case $Bt+(1+\sqrt{1-B^2})s=0$, a similar analysis lead us to find the equation \eqref{eqB}. Thus, in any case we arrive at a contradiction and the system \eqref{smk2c} has no solutions. Therefore, if $G(z)\ne -\dfrac{1}{4\pi}$ then all the critical points are nondegenerate. This finishes the proof of the claim.
\end{proof}

From the previous result it remains to study the case $G(z)=-\dfrac{1}{4\pi}$ for some $z\in\{p_1,p_2,p_3\}$, since we have that there exist three critical points of $\psi_2$ but one of them is degenerate. We shall address this difficulty by study directly the functional $\ml{F}_\la$ in that case.

\begin{claim}\label{claim3.8}
Assume that $G(z)=-\dfrac{1}{4\pi}$ for some fixed $z\in\{p_1,p_2,p_3\}$, then there exist three critical points $(z_{\la,i},m_{\la,i})$ $i=1,2,3$, of $\ml{F}_\la(z,m)$ such that, up to subsequences, as $\la\to0$ $z_{\la,i}\to z$ and $m_{\la,i}\to m_{0,i}$ with $m_{0,1}=(m_1,m_2)$, $m_{0,2}=(m_0,m_0)$ and $m_{0,3} =(m_2,m_1)$ the three different pairs of solutions to the system \eqref{smk2} obtained in Claim \ref{essmk2}.
\end{claim}

\begin{proof}[\dem] We know that $z$ is a nondegenerate critical point of $\psi_2(\cdot,m)$. This is a stable critical point situation so that for each $m\in (0,+\infty)\times(0,+\infty)$ there exist a critical point $z_\la(m)\in T$ of $\ti{\ml{F}}_\la(\cdot,m)$ with
$$\ti{\ml{F}}_\la(z,m)={1\over 8\pi\la} \lf[\ml{F}_\la(z,m)-4\pi+{\la |T| \over 2} \rg],$$
so that $z_\la(m)\to z$ as $\la \to 0$. Let us stress that $\ti{\ml{F}}_\la(z,m)= \psi_2(z,m)+O(\la|\log\la|),$ where $O$ is uniformly in $C^1$-sense for points in $T\times \Xi$. Moreover, by IFT the map $m\in\Xi\mapsto z_\la(m)$ is a $C^1$-function in $m$. Now, let us define $\ml{E}_\la(m)=\ti{\ml{F}}_\la\big(z_\la(m),m\big)$. Then, it readily follows that
\begin{align*}
\grad_m\ml{E}_\la(m)
&=\nabla_m\ti{\ml{F}}_\la\big(z_\la(m),m\big)=\nabla_m\psi_2\big(z_\la(m),m\big)+ O(\la|\log\la|),
\end{align*}
since $\nabla_z\ti{\ml{F}}_\la\big(z_\la(m),m\big)=0$. 
Hence, we look for critical points of $\ml{E}$ by study the system $\grad_m\ml{E}(m)=0$. This is equivalent to finding solutions to the perturbation of the system \eqref{smk2} given by
\begin{equation}\label{psmk2}
\lf\{\begin{array}{rcl}
(A+1)m_1+2m_1 \log m_1 -4\pi m_2 G\big(z_\la(m_1,m_2)\big) +\la|\log\la| \: \ti\theta_{\la,1}(m_1,m_2)& = &0 \\[0.1cm]
(A+1)m_2+2m_2 \log m_2  -4\pi m_1G\big(z_\la(m_1,m_2)\big) +\la|\log\la|\: \ti\theta_{\la,2}(m_1,m_2)& = &0
\end{array}\rg. ,
\end{equation}
where it holds $\ti\theta_{\la,i}=O(1)$, $i=1,2$ uniformly for $m\in\Xi$ and $4\pi G\big(z_\la(m_1,m_2)\big)\to -1$ as $\la\to 0$. Since system \eqref{smk2} with $4\pi G(z)=-1$ has exactly three different pairs of solutions, as proved in Claim \ref{essmk2}, it follows that for $\la$ small enough there are at least three different pairs $m_{\la,i}$ of solutions to \eqref{psmk2} such that as $\la\to 0$ converge, up to a subsequence, to a solution to \eqref{smk2}. Let us stress that we can consider the curves (or the implicit functions)
$$4\pi m_jG(z_\la(m_1,m_2))=(A+1)m_i+2m_i \log m_i +\la|\log\la| \: \ti\theta_{\la,i}(m_1,m_2)$$ for $i\ne j$ converging uniformly to the curves (or the implicit functions)
$$-m_j =(A+1)m_i+2m_i \log m_i $$
 for $i\ne j$ locally around each $m_0$, $m_1$ and $m_2$ in order to obtain the existence of the pairs $m_{\la,i}$, $i=1,2,3$. Thus, we conclude that $m_{\la,i}$ $i=1,2,3$ are critical points of $\ml{E}_\la$, namely, $\grad_m \ml{E}(m_{\la,i})=0$ and $m_{\la,i}\to m_{0,i}$ with $m_{0,1}=(m_1,m_2)$, $m_{0,2}=(m_0,m_0)$ and $m_{0,3} =(m_2,m_1)$ the three different pairs of solutions to the system \eqref{smk2} obtained in Claim \ref{essmk2}. By the procedure it follows that $\big(z_\la(m_{\la,i}),m_{\la,i}\big)$, $i=1,2,3$ are critical points of $\ti{\ml{F}}_\la$, in view, of
$$\grad_z\ti{\ml{F}}_\la\big(z_{\la}(m_{\la,i}), m_{\la,i}\big)=0\qquad\text{and}\qquad\nabla_m\ti{\ml{F}}_\la\big(z_\la(m_{\la,i}),m_{\la,i}\big)=\nabla_m \ml{E}_\la\big(m_{\la,i}\big)=0.$$
The proof of the Claim is finished.
\end{proof}

\medskip
In order to complete the study of existence of two bubbling solutions to \eqref{mt} in the flat two-torus in rectangular form $T$, we shall show that the multiplicity depends on the values $G(z)$ with $z\in\{p_1,p_2,p_3\}$, precisely, depends on the form of $T$. 
Recall, $G$ has three nondegenerate critical points: $p_1=\dfrac{a}{2}$, $p_2=\dfrac{\text{i}b}{2}$ (saddle points) and $p_3=\dfrac{a+\text{i}b}{2}$ (minimum point). Notice that since $G$ has zero average and $p_3$ is a minimum point we have that $G(p_3)<0$ for any $a,b>0$. In other words, $G\Big(\dfrac{a+\text{i} b}{2}\Big)\le G\Big(\dfrac{a}{2}\Big)$ and
$G\Big(\dfrac{a+\text{i} b}{2}\Big)\le G\Big(\dfrac{\text{i}b}{2}\Big)$. From an explicit formula for $G$ shown in \cite{ChO}, direct computations lead us to get that
$$G\lf(p_i\rg)= f_i\Big(\frac{b}{a}\Big),\quad i=1,2,3
$$
where $f_i$, $i=1,2,3$ are given by
\begin{equation*}
f_1(\tau)= \frac{\tau}{12 }-{1\over 2\pi}\log 2-\frac{1}{\pi}\sum_{n=1}^{+\infty}\log \lf(1+e^{-2\pi  n\tau} \rg),
\end{equation*}
\begin{equation*}
f_2(\tau)= -\frac{\tau}{24 }- \frac{1}{\pi}\sum_{n=0}^{+\infty}\log \lf(1-e^{-\pi (2n+1) \tau} \rg)
\end{equation*}
and
\begin{equation*}
f_3(\tau)= -\frac{\tau}{24 }-\frac{1}{\pi}\sum_{n=0}^{+\infty}\log \lf(1+e^{-\pi  (2n+1)\tau} \rg),
\end{equation*}
so that we can study them in terms of $\tau=\dfrac{b}{a}$. 
By symmetry arguments it follows that in case $a=b$, namely, $\tau=1$, it holds $G\Big(\ds{a\over 2}\Big)=G\Big({\text{i}a\over2}\Big)$, so, equivalently $f_1(1)=f_2(1)\approx -0.03$. By studying $f_i$, $i=1,2,3$ we obtain the following fact.

\begin{claim}
There exist $\tau_0<1<\tau_1$ such that $f_1(\tau_1)=0$, $f_2(\tau_0)=0$. If $\tau\in(0,\tau_0]\cup [\tau_1,+\infty)$ then there exist seven different critical points $(\xi_{\la,i},m_{\la,i})$ $i=1,\dots,7$ of $\ml{F}_\la$. If $\tau\in(\tau_0,\tau_1)$ then there exist nine different critical points $(\xi_{\la,i},m_{\la,i})$ $i=1,\dots,9$ of $\ml{F}_\la$.
\end{claim}
\begin{proof}[\dem]
From the definition of $f_1$ it follows that $f_1$ is a continuous function, strictly increasing and strictly concave, $f_1(0)\approx -1.43$ and $f_1(\tau)\to+\infty$ as $\tau\to+\infty$, so that, there is $\tau_1>1$ such that $f_1(\tau_1)=0$. Also, $f_2$ is a continuous function, strictly decreasing and strictly convex, with $f_2(\tau)\to+\infty$ as $\tau\to 0^+$ and $f_2(\tau)\to -\infty$ as $\tau\to+\infty$ so that there is $\tau_0<1$ such that $f_2(\tau_0)=0$. For $f_3$ we obtain that it is a continuous function, strictly concave, $\tau=1$ is a maximum with $f_3(1)\approx -0.06$, $f_3(0)\approx -1.32$ and $f_3(\tau)\to-\infty$ as $\tau\to+\infty$.

Hence, depending on the value of $\tau=\dfrac{b}{a}$ we have three cases.
\begin{enumerate}
\item If $\tau\in(0,\tau_0]$ then $f_1(\tau)=G\Big(\dfrac{a}{2}\Big)<0$, $f_2(\tau)=G\Big(\dfrac{\text{i} b}{2}\Big)\ge 0$ and $f_3(\tau)=G\Big(\dfrac{a+\text{i} b}{2}\Big)<0$, so that we have seven critical points of $\ti{\ml{F}}_\la$. Precisely, $G\Big(\dfrac{a}{2}\Big)<0$, $G\Big(\dfrac{\text{i} b}{2}\Big)\ge 0$ and $G\Big(\dfrac{a+\text{i} b}{2}\Big)<0$ gives rise to three, one and three three critical points respectively.

\item If $\tau\in(\tau_0,\tau_1)$ then $G\Big(\dfrac{a}{2}\Big)<0$, $G\Big(\dfrac{\text{i} b}{2}\Big)<0$ and $G\Big(\dfrac{a+\text{i} b}{2}\Big)<0$, so that we have nine critical points of $\ti{\ml{F}}_\la$. Precisely, $G\Big(\dfrac{a}{2}\Big), G\Big(\dfrac{\text{i} b}{2}\Big), G\Big(\dfrac{a+\text{i} b}{2}\Big)<0$ gives rise to three critical points each one.

\item If $\tau\in(\tau_1,+\infty)$ then $G\Big(\dfrac{a}{2}\Big)\ge 0$, $G\Big(\dfrac{\text{i} b}{2}\Big)<0$ and $G\Big(\dfrac{a+\text{i} b}{2}\Big)<0$, so that we have seven critical points of $\ti{\ml{F}}_\la$. Precisely, $G\Big(\dfrac{a}{2}\Big)\ge 0$, $G\Big(\dfrac{\text{i} b}{2}\Big), G\Big(\dfrac{a+\text{i} b}{2}\Big)<0$ gives rise to one, three and three three critical points respectively.
\end{enumerate}
The proof of the Claim is finished
\end{proof}
Therefore, given a rectangle we can obtain exactly either seven or nine different family of solutions. This completes the proof.
\end{proof}

\begin{proof}[{\bf Proof (of Theorem \ref{s2}):}] Assume that $S=\mathbb{S}^2$. By invariance under rotations it follows that $H(\xi,\xi)$ is constant. Furthermore, as it was said in the introduction problem \eqref{mt} is invariant under rotations, so we look for solutions with one bubbling point fixed. Thus, with a slightly abuse of notation, we are reduced to look for critical points of the functional
$$\ml{F}_\lambda (\xi_1 , m) =4\pi -{\la|T|\over 2}+8\pi\la\lf[A\sum_{j=1}^2 m_j^2+2\sum_{j=1}^2 m_j^2\log m_j-8\pi m_1 m_2G(\xi_1,\xi_2)\rg]+ \theta_\lambda(\xi_1, m ),$$
where $A=\log 16-2-8\pi H(\xi ,\xi)$. In other words, we fix $\xi_2\in\mathbb{S}^2$ and look for critical points on $\xi_1$ of $G(\cdot,\xi_2)$. Since $G(\cdot,\xi_2)$ has a global minimum, for any $m$ there exist $\xi_{1,\la}(m)\in\mathbb{S}^2$ such that $\ml{F}_{\la}$ attains its minimum at $\xi_1=\xi_{1,\la}(m)$. Hence, from the same procedure used in Claim \ref{claim3.8} it follows that for $\la$ small enough there exists $m_\la$ such that $\big(\xi_{1,\la}(m_\la),m_\la\big)$ is actually a critical point of $\ml{F}_\lambda (\xi_1 , m)$. This finishes the proof.
\end{proof}

\section{Proof of Proposition \ref{fullexpansionenergy}} \label{sec4}
\noindent The purpose of this section is to give a proof of the Proposition \ref{fullexpansionenergy}, namely, an asymptotic
expansion of the ``reduced energy" $\ml{F}_\la (\xi,m)=J_\la\big(U(\xi,m)+\ti\phi(\xi,m)\big)$, where $J_\la$ is
the energy functional given by \eqref{energy}, $U=\sqrt{\la}V$ with $V$ defined by \eqref{ansatz} and $\ti\phi=\sqrt{\la}\phi$ with $\phi$ the solution given by Proposition \ref{lpnlabis}. The proof will be divided into several steps. To this aim, we recall the following result. See \cite{EsFi} for a proof.
\begin{lem}\label{ieuf}
Letting $\bar f\in C^{2,\gamma}(S)$ (possibly depending in $\xi$),
$0<\gamma<1$. The following expansions do hold as $\delta \to 0$:
\begin{eqnarray*}
\int_S \chi_\xi e^{-\varphi_\xi} \bar f(x) e^{U_{\de,\xi}} dv_g= 8\pi \bar f(\xi)-4\pi \de^2\log\de \Delta_g \bar f (\xi) +O(\de^2),
\end{eqnarray*}
\begin{eqnarray*}
\int_S \chi_\xi e^{-\varphi_\xi} \bar f(x)
e^{U_{\de,\xi}}\frac{dv_g}{\delta^2+|y_\xi(x)|^2} =
\frac{4\pi}{\delta^2}\bar f(\xi)+\pi \Delta_g \bar f(\xi)+O(\delta^{\gamma})
\end{eqnarray*}
and
\begin{eqnarray*}
\int_S \chi_\xi e^{-\varphi_\xi} \bar f(x) e^{U_{\de,\xi}}\frac{a
\delta^2-|y_\xi(x)|^2}{(\delta^2+|y_\xi(x)|^2)^2}\, dv_g =\frac{4
\pi}{3 \de^2}(2a-1) \bar f(\xi)+(a-2)\frac{\pi}{3} \Delta_g \bar f(\xi)
+O(\delta^\gamma)
\end{eqnarray*}
for $a \in \mathbb{R}$.
\end{lem}

\medskip \noindent We are now ready to establish the expansion of $J_\la(U)$:
\begin{claim} \label{expansionenergy}
 The following expansion does
hold
\begin{equation} \label{JUt}
J_\lambda (U) =2\pi k -{\lambda |S| \over 2}+ 8\pi \la\psi_k(\xi,m)+O(\la^2|\log\la|^2)
\end{equation}
in $C(\Xi\times\ml{M})$ as $\la\to 0^+$.
\end{claim}
\begin{proof}[\dem]
First, let us consider the term
$$\int_S |\grad V|_g^2 dv_g= \int_S V (-\Delta_g V)dv_g= \sum_{j,l=1}^km_jm_l \int_S \chi_j e^{-\varphi_j} e^{U_j}PU_l dv_g$$
in view of $\ds \int_S V dv_g=0$. Since by (\ref{green}) and
(\ref{ePu})
\begin{eqnarray} \label{tricky}
\int_S \chi_j e^{-\varphi_j} e^{U_j} G(x,\xi_l) dv_g= \int_S
(-\Delta_g PU_j) G(x,\xi_l) dv_g= PU_j(\xi_l)
\end{eqnarray}
for all $j,l=1,\dots,m$, by Lemmata \ref{ewfxi}, \ref{ieuf} and
(\ref{tricky}) we have that for $l=j$
\begin{eqnarray*}
&& \int_S \chi_j e^{-\varphi_j}e^{U_j}PU_j dv_g\\
&&= \int_S \chi_j e^{-\varphi_j} e^{U_j}\lf[\chi_j(U_j-\log(8\mu_j^2\e_j^2))+8\pi H(x,\xi_j)+O(\e_j^2|\log\e_j|)\rg]dv_g\\
&&= \int_S \chi_j e^{-\varphi_j} e^{U_j}\lf[\chi_j \log \frac{|y_{\xi_j}(x)|^4}{(\mu_j^2\e_j^2+|y_{\xi_j}(x)|^2)^2}+8\pi G(x,\xi_j)\rg]dv_g +O(\e_j^2 |\log \e_j|)\\
&&=8 \int_{B_{2r_0/\mu_j\e_j}(0)} \frac{\chi^2(\mu_j\e_j |y|)
}{(1+|y|^2)^2} \log \frac{|y|^4}{(1+|y|^2)^2} dy
+8\pi PU_j(\xi_j) +O(\e_j^2|\log\e_j|)\\
&&=-16\pi -32 \pi \log \mu_j\e_j +64\pi^2 H(\xi_j,\xi_j)+O(\e_j^2
|\log \e_j|)
\end{eqnarray*}
in view of
\begin{eqnarray*}
\int_{\mathbb{R}^2}{dy \over(1+|y|^2)^2}\log{|y|^4\over
(1+|y|^2)^2} =  2 \pi\int_0^\infty{ds \over(1+s)^2}\log {s \over
1+s} =- 2 \pi  \int_0^\infty {ds \over (1+s)^2}=- 2\pi
\end{eqnarray*}
by means of an integration by parts. Similarly, by Lemmata
\ref{ewfxi}, \ref{ieuf}  and (\ref{tricky}) we have that for $l
\not= j$
\begin{eqnarray*}
\int_S \chi_j e^{-\varphi_j}e^{U_j}P U_l dv_g&=& \int_S \chi_j e^{-\varphi_j} e^{U_j}\lf[8\pi G(x,\xi_l)+O(\e_l^2|\log\e_l|)\rg]dv_g\\
&=& 64\pi^2 G(\xi_l,\xi_j) +O(\e_j^2|\log\e_j|+\e^2_l|\log\e_l|).
\end{eqnarray*}
By using the definition of $\mu_j$ and $\e_j$ and summing up the
two previous expansions, for the gradient term we get that
\begin{eqnarray*}
&&{1\over 2}\int_S  |\grad U|_g^2 dv_g= {\la\over 2}\int_S|\grad
V|_g^2\,dv_g\\
&&=2\pi k-8\pi \la(1-\log
8)\sum_{j=1}^km_j^2+8\pi\la\sum_{j=1}^km_j^2\log(2m_j^2)-32\pi^2\la\sum_{j=1}^km_j^2H(\xi_j,\xi_j)\\
&&\quad-32\pi^2\la\sum_{i\ne j}m_im_j
G(\xi_i,\xi_j)+\la\sum_{j=1}^k
O(\e_j^2|\log\e_j|).
\end{eqnarray*}

\medskip \noindent Let us now expand the potential term in $J_\lambda(U)$. By Lemma \ref{ewfxi} for any $j=1,\dots,k$ we find that $PU_j=O(\big|\log|\log\e_j|\big|)=O(|\log\la|)$, in $B_{r_0}(\xi)\sm B_{\de \over|\log\e_j|}(\xi_j).$ Recall that $PU_j=O(1)$ in $S\sm\cup_{j=1}^kB_{r_0}(\xi_j)$ for each $j=1,\dots,k$. Hence, we have that $V=O(|\log\la|)$, in $S\sm \cup_{j=1}^kB_{\de\over|\log\e_j|}(\xi_j)$
and also,
$$\int_Se^{\la V^2}\,dv_g=\sum_{j=1}^k\int_{B_{\de|\log\e_j|^{-1}}(\xi_j)}e^{\la V^2}dv_g+|S|+O(\la|\log\la|^2).$$
Now, we write
$$\int_{B_{\de|\log\e_j|^{-1}}(\xi_j)}e^{\la V^2}dv_g=\bigg[\int_{A_{\de|\log\e_j|^{-1},\de\sqrt{\e_j}}(\xi_j)}+\int_{A_{\de\sqrt{\e_j},\de\e_j|\log\e_j|}(\xi_j)}+\int_{B_{\de\e_j|\log\e_j|}(\xi_j)}\bigg]e^{\la V^2}dv_g$$
In $A_{\de|\log\e_j|^{-1},\de\sqrt{\e_j}}(\xi_j)$, we know that uniformly $V(x)=-4m_j\log|y_{\xi_j}(x)|+O(1)$. Thus, we find that
\begin{equation*}
\begin{split}
\int_{A_{\de|\log\e_j|^{-1},\de\sqrt{\e_j}}(\xi_j)} e^{\la V^2}\,dv_g&=\int_{A_{\de|\log\e_j|^{-1},\de\sqrt{\e_j}}(\xi_j)}e^{\la m_j^2[16\log^2|y_{\xi_j}(x)|+O(|\log|y_{\xi_j}(x)||)]}dv_g\\
&=O\lf(\int_{B_{\de|\log\e_j|^{-1}}(0)\sm B_{\de\sqrt{\e_j}}(0)}e^{16\la m_j^2\log|y|}dy\rg)=O(\la),
\end{split}
\end{equation*}
in view of $e^{\hat\varphi_j}=O(1)$, $\la m_j^2|\log|y||=O(1)$ in the considered region and
\begin{equation*}
\begin{split}
\int_{B_{\de|\log\e_j|^{-1}}(0)\sm B_{\de\sqrt{\e_j}}(0)}e^{16\la m_j^2\log|y|}dy&=2\pi \int_{\de\sqrt{\e_j}}^{\de|\log\e_j|^{-1}} e^{16\la m_j^2\log^2 s}s\,ds\\
&=2\pi\int_{\log(\de\sqrt{\e_j})}^{\log(\de|\log\e_j|^{-1})}e^{2t+16\la m_j^2t^2}\,dt\quad(t=\log s)\\
&=O\lf(\int_{\log(\de\sqrt{\e_j})}^{\log(\de|\log\e_j|^{-1})}e^{t}\,dt\rg)=O(\la).
\end{split}
\end{equation*}
Now, we shall use that in $B_{r_0}(\xi_j)$ it holds that $e^{\la V^2}=2m_j^2\e_j^{-2}e^{w_j+\theta_j+\la m_j^2(w_j+\theta_j)^2}$ and $\la w_j=O(1)$ uniformly. Furthermore, we know that $\ds w_j(1+\la m_j^2w_j)\le {3\over4}w_j+\beta$, where $\beta$ is a constant in $A_{\de\sqrt{\e_j},\de\e_j|\log\e_j|}$. Hence, we obtain that
$$\int_{A_{\de\sqrt{\e_j},\de\e_j|\log\e_j|}}e^{\la V^2}=O\bigg(\int_{A_{\de\sqrt{\e_j},\de\e_j|\log\e_j|}}\e_j^{-2}e^{w_j+\la m_j^2w_j^2}\bigg)=O(\la).$$
Also, it follows that
\begin{equation*}
\begin{split}
\int_{B_{\de\e_j|\log\e_j|}(\xi_j)} e^{\la V^2}\,dv_g&=2m_j^2\int_{B_{\de\e_j|\log\e_j|}(\xi_j)}\e_j^{-2}e^{w_j+\theta_j+O(\la |\log\la|^2)}dv_g\\
&=2m_j^2[8\pi +O(\la|\log\la|^2),
\end{split}
\end{equation*}
in view of $w_j=O(|\log\la|)$ in $B_{\de\e_j|\log\e_j|}(\xi_j)$ and
$$\int_{B_{\de\e_j|\log\e_j|}(\xi_j)}\e_j^{-2}e^{w_j+\theta_j}=8\pi +O\bigg(\la+\sum_{j=1}^k\e_j^2|\log\e_j|\bigg).$$
Therefore, we conclude that
\begin{equation}\label{iev2}
\int_Se^{\la V^2}\,dv_g=16\pi \sum_{j=1}^km_j^2+|S|+O(\la|\log\la|^2)
\end{equation}
and consequently,
\begin{equation*}
\begin{split}
J_\la(U)&=2\pi k+8\pi \la(\log
16- 1)\sum_{j=1}^km_j^2+8\pi\la\sum_{j=1}^km_j^2\log(m_j^2)-32\pi^2\la\sum_{j=1}^km_j^2H(\xi_j,\xi_j)\\
&\quad-32\pi^2\la\sum_{i\ne j}m_im_j
G(\xi_i,\xi_j)+\la\sum_{j=1}^k
O(\e_j^2|\log\e_j|)-{\la\over2}\lf[16\pi \sum_{j=1}^km_j^2+|S|+O(\la|\log\la|^2)\rg].
\end{split}
\end{equation*}
This completes the proof.
\end{proof}

In order to expand the derivatives $\fr_{(\xi_l)_q} \ml{F}_\la$ and $\fr_{m_l} \ml{F}_\la$, and get some estimates for $\fr_{(\xi_l)_q} \phi$ and $\fr_{m_l} \phi$, we have to expand $\fr_{(\xi_l)_q} V$ and $\fr_{m_l} V$ for $q=1,2$  and $l=1,\dots,k$. Let us notice that from the definition of $PU_j$ and similar ideas to obtain the expansion \eqref{eaxi}, we have that the derivatives $\fr_{(\xi_l)_q} PU_j$, for $q=1,2$ and $\fr_{m_l} PU_j$ for $ j,\, l=1,\dots, k$ expand as follows
\begin{equation}
\begin{split}
\fr_{(\xi_l)_q}PU_j(x)=&\,\de_{jl} \fr_{(\xi_l)_q}\chi_j[U_j-\log(8\mu_j^2\e_j^2)] + \chi_j \fr_{(\xi_l)_q}[U_j-\log(8\mu_j^2\e_j^2)] \\
&\,+ \de_{ jl} 8\pi \fr_{(\xi_l)_q}H(x, \xi_j) + O\lf( \e_j^2|\log\e_j|\rg)
\end{split}
\end{equation}
and
\begin{equation}
\fr_{m_l}PU_j(x)= \chi_j \fr_{m_l}[U_j-\log(8\mu_j^2\e_j^2)] + O\lf( \e_j^2|\log\e_j|^2\rg),
\end{equation}
uniformly in $S$. Let us stress that $\mu_j=\mu_j(\xi,m)$ and $\e_j=\e_j(m_j)$. Furthermore, from the definition of $\e_l$ we get
\begin{equation}\label{del}
\fr_{m_l} \e_l={\e_l\over 4}\lf({1\over \la m_l^3 } + {4\over m_l} \rg)={\e_l\over m_l}\lf[ -2\log\e_l+\log(2m_l^2)+1\rg].
\end{equation}
Hence, we have that uniformly in $S$
\begin{equation}\label{dxiv}
\begin{split}
\fr_{(\xi_l)_q}V(x)
=&\,-2m_l\fr_{(\xi_l)_q}\chi_l\log(\mu_l^2\e_l^2+|y_{\xi_l}(x)|^2) - m_l \chi_l {4\mu_l\fr_{(\xi_l)_q}\mu_l\e_l^2+2\fr_{(\xi_l)_q}(|y_{\xi_l}(x)|^2)\over \mu_l^2\e_l^2+|y_{\xi_l}(x)|^2} \\
&\,+ 8\pi m_l \fr_{(\xi_l)_q}H(x, \xi_l)-\sum_{j\ne l} m_j\chi_j {4\mu_j\fr_{(\xi_l)_q}\mu_j\e_j^2\over \mu_j^2\e_j^2+|y_{\xi_j}(x)|^2} + \sum_{j=1}^kO\lf( \e_j^2|\log\e_j|\rg)\\
\end{split}
\end{equation}
and
\begin{equation}\label{dmv}
\begin{split}
\fr_{m_l }V(x)
=&\,-2 \chi_l\log(\mu_l^2\e_l^2+|y_{\xi_l}(x)|^2) + 8\pi  H(x, \xi_l)- m_l \chi_l {4\mu_l\fr_{m_l}\mu_l\e_l^2+4\mu_l^2\e_l\fr_{m_l }\e_l\over \mu_l^2\e_l^2+|y_{\xi_l}(x)|^2} \\
&\,-\sum_{j\ne l} m_j\chi_j {4\mu_j\fr_{ m_l }\mu_j\e_j^2\over \mu_j^2\e_j^2+|y_{\xi_j}(x)|^2} + \sum_{j=1}^kO\lf( \e_j^2|\log\e_j|^2\rg).
\end{split}
\end{equation}

\begin{claim} \label{expansionenergydm}
The following expansion does
hold
\begin{equation} \label{dmJUt}
\fr_{m_l}[J_\lambda (U)] =8\pi  \la \fr_{m_l}\psi_k(\xi,m)+O(\la^2|\log\la|^2)
\end{equation}
in $C(\Xi\times\ml{M})$ as $\la\to 0^+$.
\end{claim}
\begin{proof}[\dem]Notice that
$$\fr_{m_l}[J_\lambda (U)] =-\la\int_S\lf[\lap_gV+\la Ve^{\la V^2}\rg]\fr_{m_l}V\, dv_g.$$
Hence, we have that
$$\int_S\lap_gV \,\fr_{m_l}V\, dv_g=- \sum_{i=1}^km_i\int_S\chi_ie^{-\varphi_i}e^{U_i} \fr_{m_l}V$$
in view of $\ds\int_S \fr_{m_l}V=0$. For $i\ne l$, by using Lemma \ref{ieuf} and \eqref{dmv}, we find that
\begin{equation*}
\begin{split}
\int_S \chi_ie^{-\varphi_i}e^{U_i}\,\fr_{m_l}V\, dv_g=&\,\int_{S}\chi_ie^{-\varphi_i}e^{U_i} \bigg[ 8\pi  G(x, \xi_l) - m_i\chi_i {4\mu_i\fr_{ m_l }\mu_i\e_i^2\over \mu_i^2\e_i^2+|y_{\xi_i}(x)|^2} + \sum_{j=1}^kO\lf( \e_j^2|\log\e_j|^2\rg)\bigg]\\
=&\, 	 64\pi^2G(\xi_i,\xi_l)-16\pi m_i{\fr_{m_l}\mu_i\over \mu_i}+\sum_{j=1}^kO\lf( \e_j^2|\log\e_j|^2\rg)=\sum_{j=1}^kO\lf( \e_j^2|\log\e_j|^2\rg),
\end{split}
\end{equation*}
in view of $\chi_j\equiv 0$ in $B_{2r_0}(\xi_i)$ with $j\ne i$ and  from \eqref{muj}
\begin{equation}\label{dmujl}
2{\fr_{m_l}\mu_j\over \mu_j}=8\pi m_j^{-1}G(\xi_l,\xi_j)\quad\text{with $j\ne l$.}
\end{equation}
Also, using again \eqref{dmv}, we find that
\begin{equation*}
\begin{split}
\int_S \chi_le^{-\varphi_l}e^{U_l} \,\fr_{m_l}V\, dv_g=&\, \int_S\chi_le^{-\varphi_l}e^{U_l} \bigg[-2 \chi_l\log(\mu_l^2\e_l^2+|y_{\xi_l}(x)|^2) + 8\pi  H(x, \xi_l)\\
&\,\qquad\qquad\qquad- m_l \chi_l {2 \fr_{m_l}(\mu_l^2\e_l^2)\over \mu_l^2\e_l^2+|y_{\xi_l}(x)|^2} + \sum_{j=1}^kO\lf( \e_j^2|\log\e_j|^2\rg)\bigg]\\
=&\, 	-16\pi-32\pi \log(\mu_l\e_l)+64\pi^2H(\xi_l,\xi_l)-8\pi m_l\pi{\fr_{m_l}(\mu_l^2\e_l^2)\over \mu_l^2\e_l^2}\\
&\,+\sum_{j=1}^kO\lf( \e_j^2|\log\e_j|^2\rg).
\end{split}
\end{equation*}
From the definition \eqref{muj}-\eqref{e}, we get that
\begin{equation}\label{dmulel}
2{\fr_{m_l}\mu_l\over \mu_l}=-{4\over m_l}-{8\pi\over m_l^{2}}\sum_{i\ne l} m_iG(\xi_l,\xi_i)\qquad\text{and}\qquad
-{4\over \e_l}\fr_{m_l}\e_l=-{1\over \la m_j^3}-{4\over m_j}
\end{equation}
so that, we conclude that
\begin{equation*}
\begin{split}
\int_S\lap_gV \,\fr_{m_l}V\, dv_g=&\, - 16\pi \log 8 m_l - 16\pi m_l \log (2 m_l^2) +64\pi^2 m_lH(\xi_l,\xi_l)\\
&\,+64\pi^2\sum_{j=1,j\ne l}^km_j G(\xi_j,\xi_l)+O(\la|\log\la|).
\end{split}
\end{equation*}

On the other hand, combining arguments to deduce \eqref{isvev2} and \eqref{iev2}, using Lemma \ref{ieuf} and the derivatives \eqref{dmv} and \eqref{dmujl}-\eqref{dmulel},  we find that
$$\int_S \la Ve^{\la V^2}\fr_{m_l}V\, dv_g=16\pi m_l+O(\la|\log\la |).$$
Thus, we conclude \eqref{dmJUt}.
\end{proof}

\begin{claim}
The following expansion does hold
\begin{equation} \label{fullJUt0}
\ml{F}_\lambda (\xi , m) = J_\la(U)+\ti\theta_\la(\xi,m)
\end{equation}
in $C(\Xi\times \ml{M})$ and $C^1(\ml{M})$ as $\la\to 0^+$, where the term $\ti\theta_\lambda(\xi, m )$ satisfies
\begin{eqnarray} \label{rlambda}
|\ti\theta_\la(\xi,m)| +\sum_{l=1}^k \la \lf|\fr_{m_l}\ti \theta_\la(\xi,m)\rg|=O(\la^3)
\end{eqnarray}
as $\la\to0^+$ uniformly for points $(\xi, m )\in \Xi\times\ml{M}$.
\end{claim}

\begin{proof}[\dem]
Taking into account that $DJ_\la(U +
\ti \phi)[\ti\phi]=0$, a Taylor expansion, $\sqrt{\la}\, V= U $ and $\sqrt{\la }\, \phi (\xi,m) = \ti\phi(\xi,m)$, the definition of $K$ and \eqref{pnlabis} gives us
\begin{equation}\label{A}
\begin{split}
J_\la(U+\ti\phi)-J_\la(U)
& = - \int_0\p1 D^2J_\la\big(U+t\ti \phi \big)\big[\ti\phi,\ti\phi\big]\,t\,dt,\\
&=-\la\int_0\p1  \lf( \int_{S} [R + N(\phi)]\, \phi  +  \int_{S}
 \lf[ K - f'(V+t\phi) \rg]\phi^2 \rg)\,t\,dt. 
\end{split}
\end{equation}
Therefore, we get
$$
J_\la(U+\ti \phi)-J_\la(U)=O(\la^3),
$$
since $\|R\|_*\le C\la$, $\| N(\phi)\|_*\le
C[\la\|\phi\|_\infty+ \|\phi\|_\infty^2]$, for some $s\in (0,t)$
$$\|K - f'(V+t \phi)\|_*\le \|f''(V+ s \phi)\, \phi\|_*+\|K-f'(V)\|_*\le C[\|\phi\|_\infty+\la]$$
and $\|\phi\|_\infty\le C\la$.

Let us differentiate with respect to $\be=m_l$. We use
representation \eqref{A} and differentiate directly under the
integral sign, thus obtaining, for each $l=1,\dots,k$
\begin{equation*}
\begin{split}
\fr_\be\big[J_\la&(U+\ti\phi)-J_\la(U) \big]=-\la\int_0\p1  \lf( \int_{S} \lf\{[\fr_\be R + \fr_\be N(\phi)]\, \phi +[R + N(\phi)]\, \fr_\be \phi\rg\}   \rg)\,t\,dt\\
 &\quad\qquad \qquad-\la\int_0\p1  \lf( \int_{S}  \lf\{ \fr_\be K - \fr_\be[f'(V+t\phi) ] \rg\} \phi^2  +  \lf[ K - f'(V+t\phi) \rg]2\phi\,\fr_\be \phi \rg)\,t\,dt.
\end{split}
\end{equation*}
Using Proposition \ref{lpnlabis} and the computations in the Appendix B, we get that
\begin{equation*}
\begin{split}
\big|\fr_{m_l}\big[J_\la&(U+\ti\phi)-J_\la(U) \big] \big|
\le C\la\bigg( \lf[\la\log\e_l + \la^2\log\e_l + \la\|\fr_{m_l}\phi\|_\infty \rg]\, \|\phi\|_\infty \\
& \quad\qquad\qquad+\lf[\la+ \la\|\phi\|_\infty+\|\phi\|_\infty^2\rg]\,\| \fr_{m_l} \phi\|_\infty
+ \la\log \e_l \| \phi\|_\infty^2  +  \lf[ \|\phi\|_\infty+\la \rg]\, \|\phi\|_\infty\,\|\fr_{m_l} \phi\|_\infty \bigg).
\end{split}
\end{equation*}
Thus, we conclude
$$
\fr_{m_{l}} \lf[J_\la(U+\ti\phi)-J_\la(U)\rg]=O\big( \la^3\;|\log\e_l|\big)=O(\la^2),\qquad l=1,\dots,k
$$
Now, taking $\ti\theta_\la(\xi,m)=\ml{F}_\la(\xi,m)-J_\la(U)$, we have shown \eqref{rlambda} as $\la\to0^+$. The continuity in $(\xi,m)$ of all these expressions is inherited from that of $\phi$ and its derivatives in $\xi$ and $m$ in the $L^\infty$ norm.
\end{proof}

\noindent Now, we are going to study the derivatives of $\ml{F}_\la$ with respect to $\beta=(\xi_l)_q$ with $q=1,2$ and $l=1,\dots,k$. Due to the estimates \eqref{cotadphi} given in Proposition \ref{lpnlabis} we have to address this expansion in a different way. We shall use similar ideas first presented in \cite{EMP} and also used in \cite{CDF}.

\begin{claim} \label{expansionenergydxi}
The following expansion does
hold
\begin{equation} \label{dxiJUt}
\fr_{(\xi_l)_q}\ml{F}_\lambda (\xi,m) =8\pi  \la \fr_{(\xi_l)_q}\psi_k(\xi,m)+O(\la^2|\log\la|^2)
\end{equation}
in $C(\Xi\times\ml{M})$ as $\la\to 0^+$.
\end{claim}
\begin{proof}[\dem]
Let us differentiate the function $\ml{F}_\la(\xi,m)$ with respect to $(\xi_l)_q$ with $q=1,2$ and $l=1,\dots,k$. Since $\sqrt\la \,V(\xi,m) =U(\xi,m)$ and $\sqrt\la\,\phi(\xi,m) = \ti\phi(\xi,m)$, we can differentiate directly $J_\la\big(\sqrt\la[V+\phi] \big)$ (under the integral sign), so that integrating by parts we get
\begin{equation*}
\begin{split}
\fr_{(\xi_l)_q}\ml{F}_\la(\xi,m) 
=&\,-\la\int_S\lf[\Delta_g(V+\phi) + \la( V+\phi) e^{\la(V+\phi)^2}\rg]\,\lf[\fr_{(\xi_l)_q}V + \fr_{(\xi_l)_q}\phi\rg]\,dv_g\\
=&\, -\la \sum_{i=0}^{2}
\sum_{j=1}^k c_{ij} \int_{S} \Delta_gPZ_{ij}\,\lf[\fr_{(\xi_l)_q}V + \fr_{(\xi_l)_q}\phi\rg]dv_g,
\end{split}
\end{equation*}
since $\ds \int_{S} \big(\fr_{(\xi_l)_q}V + \fr_{(\xi_l)_q}\phi\big) dv_g=0$. From the orthogonal conditions we find that for $j\ne l$
\begin{equation*}
c_{ij} \int_{S} \Delta_gPZ_{ij}\,\fr_{(\xi_l)_q}\phi\, dv_g=- c_{ij}\int_{S} \fr_{(\xi_l)_q}\lf[\Delta_gPZ_{ij}\rg]\,\phi\, dv_g= O\lf(\max_{ij}|c_{ij}|\, \|\phi\|_\infty\rg)=O(\la^2)
\end{equation*}
in view of $\ds\int_S\rho(x)\, dv_g\le C$, $\|\fr_{(\xi_l)_q}\Delta_gPZ_{ij}\|_*\le C$ for $j\ne l$ and \eqref{cotaphi}. For $j=l$, we compare $\fr_{(\xi_l)_q}\Delta_gPZ_{il}$ with derivatives $\fr_{x_q}\Delta_gPZ_{il}$ to get
\begin{equation*}
 \int_{S} \Delta_gPZ_{il}\,\fr_{(\xi_l)_q}\phi\, dv_g=- \int_{S} \fr_{(\xi_l)_q}\lf[\Delta_gPZ_{il}\rg]\,\phi\, dv_g= \int_{S} \fr_{x_q}\lf[\Delta_gPZ_{il}\rg]\,\phi\, dv_g +O(\la)
\end{equation*}
Thus, integrating by parts we deduce that
\begin{equation*}
\begin{split}
\sum_{i=0}^{2}
\sum_{j=1}^k c_{ij} \int_{S} \Delta_gPZ_{ij}\,\fr_{(\xi_l)_q}\phi\, dv_g
&=-\sum_{i=0}^{2}
 c_{il}\int_{S} \Delta_gPZ_{il}\,\fr_{x_q}\phi\, dv_g+ O\lf(\la^2\rg)\\
&= -\int_{B_{2r_0}(\xi_l)}\lf[\Delta_g(V+\phi) + \la( V+\phi) e^{\la(V+\phi)^2}\rg]\,\fr_{x_q}\phi\,dv_g+O(\la^2),
\end{split}
\end{equation*}
On the other hand, from \eqref{dxiv}, we obtain that for $j\ne l$ in $B_{2r_0}(\xi_j)$
$$\fr_{(\xi_l)_q}V(x)=8\pi m_l \fr_{(\xi_l)_q}G(x, \xi_l)-m_j\chi_j {4\mu_j\fr_{(\xi_l)_q}\mu_j\e_j^2\over \mu_j^2\e_j^2+|y_{\xi_j}(x)|^2} + \sum_{i=1}^kO\lf( \e_i^2|\log\e_i|\rg)$$
and consequently
\begin{equation*}
\begin{split}
\int_{S} \Delta_gPZ_{ij}\,\fr_{(\xi_l)_q} V \, dv_g&=- \int_{B_{2r_0}(\xi_j)} \chi_j e^{-\varphi_j}\e_j^{-2}e^{w_j}Z_{ij} \bigg[8\pi m_l \fr_{(\xi_l)_q}G(x, \xi_l)-m_j\chi_j {4\mu_j\fr_{(\xi_l)_q}\mu_j\e_j^2\over \mu_j^2\e_j^2+|y_{\xi_j}(x)|^2} \\
&\qquad+ \sum_{i=1}^kO\lf( \e_i^2|\log\e_i|\rg)\bigg]\,\, dv_g
\end{split}
\end{equation*}
From the definition of $Z_{ij}$ it follows that for $j\ne l$
$$\int_{S} \Delta_gPZ_{ij}\,\fr_{(\xi_l)_q} V \, dv_g=\begin{cases}
O(1)& \text{if } i=0,\\
O(\e_j)&\text{if } i=1,2.
\end{cases}$$
For $j=l$, we compare $\fr_{(\xi_l)_q}V$ with derivatives $\fr_{x_q}V$ to get
\begin{equation*}
\int_{S} \Delta_gPZ_{il}\,\fr_{(\xi_l)_q} V\, dv_g=- \int_{S} \Delta_gPZ_{il} \,\fr_{x_q}V\, dv_g+O(\la).
\end{equation*}
Hence, taking into account that $|c_{ij}|\le C\la$ for all $i=0,1,2$ and $j=1,\dots,k$, we obtain that
\begin{equation*}
\begin{split}
\sum_{i=0}^{2}
\sum_{j=1}^k c_{ij} \int_{S} \Delta_gPZ_{ij}\,\fr_{(\xi_l)_q} V\, dv_g
&=\sum_{i=0}^{2} c_{il}\int_{S}  \Delta_gPZ_{il}\, \fr_{(\xi_l)_q}V\, dv_g+O(\la)\\
&=-\sum_{i=0}^{2} c_{il}\int_{S}  \Delta_gPZ_{il}\, \fr_{x_q}V\, dv_g+O(\la)\\
&=-\int_{B_{2r_0}(\xi_l)} \lf[\Delta_g(V+\phi) + \la( V+\phi) e^{\la(V+\phi)^2}\rg]\fr_{x_q}V+O(\la).
\end{split}
\end{equation*}
Therefore, denoting $v_\xi=V+\phi$ we get that
\begin{equation*}
\begin{split}
\fr_{(\xi_l)_q}\ml{F}_\la(\xi,m)
=&\,\la\int_{B_{2r_0}(\xi_l)} \lf[\Delta_g v_\xi + \la v_\xi e^{\la v_\xi^2}\rg]\,\fr_{x_q} v_\xi\,dv_g +O(\la)
\end{split}
\end{equation*}
Hence, using a Pohozaev type of identity as used in \cite[Proof of Lemma 5.3]{EMP} or \cite[Proof of Proposition 3.2]{CDF} and the expansion
$$V(x)+\phi(x)= m_l\lf[ -4\chi_l(x)\log|y_{\xi_l}(x)| + 8\pi H(x,\xi_l)\rg]+\sum_{j=1,j\ne l}^k 8\pi G(x,\xi_j)+O(\la)$$
uniformly on compact subsets of $\bar B_{2r_0}(\xi_l)\sm \{\xi_l\}$ in $C^1$-sense we obtain the following expansion
$$\int_{B_{2r_0}(\xi_l)} \lf[\Delta_g v_\xi + \la v_\xi e^{\la v_\xi^2}\rg]\,\nabla_g v_\xi\,dv_g=8\pi \nabla_{\xi_l}\psi(\xi,m)+O(\la|\log\la|).$$
This finishes the proof.
\end{proof}

\noindent Therefore, taking into account the expansions \eqref{JUt}, \eqref{dmJUt}, \eqref{fullJUt0} and \eqref{dxiJUt} we conclude the proof of Proposition \ref{fullexpansionenergy}.


\begin{appendices}
\section{\hspace{-0.5cm}: The linear theory}\label{appeA}

\noindent In this section, we will study the linearized operator
under suitable orthogonality conditions. Throughout the main part
of this section we only assume that the numbers $\mu_j$,
$j=1,\dots,k$ satisfy
$C_0^{-1}\le\mu_j\le C_0$ for all $j=1,\dots,k$ independently of
$\la$ and that the points $\xi_j\in S$, $j=1,\dots,k$ are
uniformly separated from each other, namely,
$\xi=(\xi_1,\dots,\xi_k)\in \Xi$.

\medskip
For $\mu,\e,m>0$ define the function
\begin{equation}\label{ro}
\begin{split}
\rho_{\mu,\e,m}(y)=&\,\chi\Big({r_0|y|\over \delta\e|\log\e|^2}\Big)\Big(1+\big|w_\mu\big({y\over\e}\big)\big|+\big|w_\mu\big({y\over\e}\big)\big|^2\Big)\e^{-2} e^{w_\mu\big({y\over\e}\big)}\\
&+\Big[1-\chi\Big({2r_0|y|\over \delta\e|\log\e|^2}\Big)\Big]\Big[\big\{1+\big|\log|y|\big|\big\}e^{\la m^2w_\mu({y\over\e})}+\la^{-1}\Big]\e^{-2} e^{w_\mu({y\over\e})},
\end{split}
\end{equation}
with $y\in\R^2$, so that, $\rho_j(x)=\rho_{\mu_j,\e_j,m_j}(y_{\xi_j}(x))$ for $x\in
B_{2r_0}(\xi_j)$, with $\rho_j$ defined in \eqref{roj}.
First, we will prove the following result.

\begin{lem}
There exist positive constants $C,\e_0>0$ with $C=C(\de,r)$ such that for all $0<\e<\e_0$ the solution $\psi$ to problem
$$
\begin{array}{l}
-\Delta \psi =\ds \rho_{\mu,\e,m}, \quad
\de\e<|y| < r, \cr\cr \psi (y) = 0 \hbox{ for }|y| = \de\e,\quad |y| =r
 \end{array}$$
 satisfies the estimate $\|\psi\|_\infty\le C.$ Furthermore, $C(\de,r)$ could be smaller if we choose $\de$ large and $r$ small.
\end{lem}
\noindent To be more precise, we will need to take $\de$ large and $r$ small enough so that $2C(\de,r)+r^2<1$.
\begin{proof}[\dem]
 Since $\rho_{\mu,\e,m}$ is radial, $\psi(y)=\psi(|y|)$. If $\varphi(t)=\psi(e^t)$ then for $t\in[\log(\de\e),\log r]$ we study
 $$
\begin{array}{l}
-\varphi''(t)=\ds e^{2t}\rho_{\mu,\e,m}(e^t):=g_0(t), \qquad\varphi (\log(\de\e)) = 0, \qquad\varphi(\log r)=0.
 \end{array}$$
Direct computations shows that
$$\varphi(t)={t-\log(\de\e)\over \log{r\over \de\e}}\int_{\log(\de\e)}^{\log r}\int_{\log(\de\e)}^sg_0(\tau)\,d\tau\,ds-\int_{\log(\de\e)}^{t}\int_{\log(\de\e)}^sg(\tau)\,d\tau\,ds.$$
Notice that $\varphi$ is concave and its maximum $t_0$ satisfies
$$\varphi'(t_0)={1\over \log{r\over \de\e}}\int_{\log(\de\e)}^{\log r}\int_{\log(\de\e)}^sg_0(\tau)\,d\tau\,ds-\int_{\log(\de\e)}^{t_0}g(\tau)\,d\tau=0$$ and hence,
$$\varphi(t_0)=\int_{\log(\de\e)}^{t_0}g_0(s)(s-\log(\de\e))\,ds=\int_{t_0}^{\log r}g(s)(\log r-s)\,ds.$$
Thus, we deduce that for all $t\in[\log(\de\e),\log r]$
$$0\le\varphi(t)\le\varphi(t_0)\le\max\lf\{\int_{\log(\de\e)}^{\log(\de\sqrt{\e})}g_0(s)(s-\log(\de\e))\,ds,
\int_{\log(\de\sqrt{\e})}^{\log r}g(s)(\log r-s)\,ds\rg\}.$$
We estimate every integral in the following way
\begin{equation*}
\begin{split}
\int_{\log(\de\e)}^{\log(\de\sqrt{\e})}g_0(s)(s-\log(\de\e))\,ds&\le C\bigg[\int_{\log(\de\e)}^{\log(2\de\e|\log\e|^2)}\e^2e^{-2s}(1+|\log\e-s|+|\log\e-s|^2)(s-\log(\de\e))\,ds\\
&\qquad\quad+\int_{\log({\de\over2}\e|\log\e|^2)}^{\log(\de\sqrt{\e})}[\e\{1+s\}e^{-s}+\la^{-1}\e^{2}e^{-2s}](s-\log(\de\e))\,ds\bigg]\\
&\le C\bigg[\int_{\log\de}^{\log(2\de|\log\e|^2)}e^{-2t}(1+|t|+|t|^2)(t-\log \de)\,dt\\
&\quad\qquad+\int_{\log({\de\over2}|\log\e|^2)}^{\log({\de\over\sqrt{\e}})}(\{1+t+\log\e\}e^{-t}+\la^{-1}e^{-2t})(t-\log\de)\,dt\bigg]\\
&\le C_1(\de)[C_2(\e)+1]\quad\qquad(\text{with }t=s-\log\e)
\end{split}
\end{equation*}
with $C_1(\de),C_2(\e)\to0$ as $\de\to+\infty$ and $\e\to0$, in view of \eqref{ro},
$$\Big|w_{\mu}\Big({e^s\over\e}\Big)\Big|=|\log(8\mu^2)-2\log(1+\mu^2\e^2e^{-2s})+4(\log\e-s)|=O(1+|\log\e-s|),$$
for $s\in[\log(\de\e),\log(2\de\e|\log\e|^2)]$, $\e^{-2}e^{w_{\mu}({e^s\over\e})}=O(\e^2e^{-4s})$, for $s\in[\log(\de\e),\log(\de\sqrt{\e})]$ and \linebreak$e^{w_{\mu}({e^s\over\e})+\la m^2w_{\mu}^2({e^s\over\e})}=O(e^{{3\over4}w_{\mu}({e^s\over\e})})=O(\e e^{-3s})$, for $s\in[\log(\de\e|\log\e|^2),\log(2\de\sqrt{\e})]$. And for the second integral
\begin{equation*}
\begin{split}
\int_{\log(\de\sqrt{\e})}^{\log r}g_0(s)(\log r-s)\,ds&\le C\int_{\log(\de\sqrt{\e})}^{\log r}\lf[(1+|s|)e^{2s+16\la m^2s^2}+ \la^{-1}\e^2e^{-2s}\rg][\log r-s]\,ds\\
&\le C\int_{\log(\de\sqrt{\e})}^{\log r}\big[e^{s}(1+|s|+|s|^2)+\la^{-1}\e^{2}e^{-2s}(\log r-s)\big]\,ds\\
&\le C_3(r)[1+C_2(\e)]+C_1(\de)C_2(\e)[1+|\log r|]
\end{split}
\end{equation*}
with $C_3(r)\to0$ as $r\to0$, in view of \eqref{ro}, $\e^{-2}e^{w_{\mu}({e^s\over\e})}=O(\e^2e^{-4s})$, for $s\in[\log(\de\sqrt{\e}),\log r]$
and
$\e^{-2}e^{w_{\mu}({e^s\over\e})+\la m^2w_{\mu}^2({e^s\over\e})}=O(e^{16\la m^2s^2})$, and $e^{2s+16\la m^2s^2}=O(e^s)$, for $s\in[\log(\de\sqrt{\e}),\log r]$.
\end{proof}

\medskip
\noindent We are now ready for

\medskip

\begin{proof}[{\bf Proof (of Proposition \ref{p2}): }] The proof of estimate \eqref{estmfe1} consists of
several steps. Let assume the opposite, namely, the existence of
sequences $\la_n \to 0$, points $\xi^n=(\xi_1^n,\dots,\xi_k^n)\in
\Xi$, numbers $m_j^n$ with $m^n=(m_1^n,\dots,m_k^n) \in\ml{M}$,
$\mu_j^n$, $\e_j^n$ and $c_{ij}^n$, functions $h_n$ with $\|h_n\|_* \to 0 $ as $n\to+\infty$, $\phi_n$ with $\|\phi_n
\|_\infty =1$, and
\begin{equation}\label{ephin}
\begin{cases}
L(\phi_n) \, =\, h_n+\sum_{i=0}^2\sum_{j=1}^k c_{ij}^n\lap_gPZ^n_{ij},&\text{ in $S$},\\
\\[-0.5cm]
\int_{S} \phi_n \Delta_gPZ^n_{ij}\,dv_g = 0, \hbox{ for all}&\
i=0,1,2,j=1,\dots,m, \quad \int_{S}\phi_n\,dv_g=0,
\end{cases}
\end{equation}
Without loss of generality, we assume that $\xi^n_j\to\xi^*_j$, $m_j^n\to m_j^*$, $\mu_j^n\to\mu_j^*$ as $n\to+\infty$ and $\xi^*=(\xi^*_1,\dots,\xi_k^*)\in\Xi$ $m^*=(m_1^*,\dots,m_k^*)\in \ml{M}$. First, we have the following fact.

\begin{claim}\label{estcij}
There exists a constant $C>0$ independent of $n$ such that for all
$i=0,1,2$, $j=1,\dots,k$ it holds $|c_{ij}^n|\le
C\lf[\|h_n\|_*+\|\phi_n\|_\infty\rg]$ as $n\to+\infty$.
\end{claim}
\begin{proof}[\dem] For notational purpose we omit the index $n$. To estimate the values of the $c_{ij}$'s, test equation \equ{ephin} against $PZ_{ij}$, $i=0,1,2$ and $j=1,\dots,m$:
$$\int_S \phi (\lap_gPZ_{ij}+KPZ_{ij})\,dv_g =\int_S h PZ_{ij}\,dv_g+\sum_{p=0}^2\sum_{q=1}^k c_{pq} \int_S \lap_g
PZ_{pq} PZ_{ij}\,dv_g.$$
Since for $j=1,\dots,k$ we have the following estimates in $C(S)$
\begin{equation}\label{pzij}
PZ_{ij}=\chi_jZ_{ij}+O(\e_j)\,,\:\:\:i=1,2\,, \qquad  PZ_{0j}=\chi_j(Z_{0j}+2)+O(\e^2_j|\log\e_j|),
\end{equation}
it readily follows that
\begin{equation}\label{ilpzpz}
\int_S \lap_g PZ_{pq} PZ_{ij}dv_g=-{32\pi\over 3}\delta_{p i}\delta_{q j}+O(\e_j),
\end{equation}
 where the $\delta_{ij}$'s are the Kronecker's symbols. Furthermore, we find that
\begin{equation*}
\begin{split}
\int_S \phi \ti L(PZ_{0j})\,dv_g &=\int_S \phi \bigg(\chi_j\lap_gZ_{0j}+\sum_{l=1}^k\chi_le^{-\varphi_l}\e_l^{-2}e^{w_l}\lf[\chi_j(Z_{0j}+2)+O(\e^2_j|\log\e_j|)\rg]\bigg)\,dv_g\\ &=\int_S \phi \chi_je^{-\varphi_j}\e_j^{-2}e^{w_j} \Big(-Z_{0j}+\chi_jZ_{0j}+2\chi_j\Big)\,dv_g+O(\e^2_j|\log\e_j|\|\phi\|_\infty)\\
&=2\int_S \phi \chi_j^2e^{-\varphi_j}\e_j^{-2}e^{w_j} \,dv_g+O(\e^2_j|\log\e_j|\,\|\phi\|_\infty)
\end{split}
\end{equation*}
and similarly for $i=1,2$
\begin{equation*}
\begin{split}
\int_S \phi \ti L(PZ_{ij})\,dv_g &=\int_S \phi \bigg(\chi_j\lap_gZ_{ij}+\sum_{l=1}^k\chi_le^{-\varphi_l}\e_l^{-2}e^{w_l}\lf[\chi_jZ_{ij}+O(\e_j)\rg]\bigg)\,dv_g\\ &=\int_S \phi \chi_je^{-\varphi_j}\e_j^{-2}e^{w_j} \big(-Z_{ij}+\chi_jZ_{ij}\big)\,dv_g+O(\e_j\|\phi\|_\infty)\\
&=O(\e_j\|\phi\|_\infty)
\end{split}
\end{equation*}
Hence, we get the estimates $\ds|c_{0j}|\le C\bigg[\|h\|_*+\|\phi\|_\infty+\e_j\sum_{p=0}^2\sum_{q=1}^k
|c_{pq}|\bigg]$ and
\begin{equation}\label{estcij0}
|c_{ij}|\le
C\bigg[\|h\|_*+\e_j\|\phi\|_\infty+\e_j\sum_{p=0}^2\sum_{q=1}^k
|c_{pq}|\bigg],\qquad\text{for $i=1,2$.}
\end{equation}
Thus, the claim follows.
\end{proof}

Now, we prove the asymptotic behavior of $\phi_n$ on compact subsets of $S\sm\{\xi^*_1,\dots,\xi^*_m\}$.

\begin{claim}\label{ab}
There holds $\phi_n\to0$ as $n\to+\infty$ in $C^{1}$ uniformly
over compact subsets of \linebreak $S\sm\{\xi^*_1,\dots,\xi^*_m\}$. In
particular, given any $2r_0>r>0$ we have
\begin{equation}\label{cpe}
\|\phi_n\|_{L^\infty(S\sm
\cup_{j=1}^mB_{r}(\xi^n_j))}\to0\qquad \text{as $n\to+\infty$}.
\end{equation}
\end{claim}

\begin{proof}[\dem] Note that for any $0<r<r_0$ it holds that up to a subsequence
as $n\to+\infty$
$$
c(\phi_n)=-\frac{1}{|S|}\int_{S}K(x)\phi_n(x)\,dv_g
=-\frac{1}{|S|}\sum_{j=1}^k\int_{B_{r}(\xi_j)}e^{-\varphi_j}(\e_j^{n})^{-2}
e^{w_j}\phi_n\,dv_g+O([\e_j^n]^2)=c_0+ o(1).$$ From Claim
\ref{estcij}, it readily follows that $\sum_{i=0}^2\sum_{j=1}^k
c_{ij}^n\lap_gPZ^n_{ij}=o(1)$ as $n\to+\infty$ in $S\sm
\cup_{j=1}^m B_{r}(\xi^n_j)$ for a given $r>0$, in
view of $\lap_gPZ^n_{ij}=O([\e_j^n]^2)-\frac{1}{|S|}\int_S
\chi_j\Delta_g Z^n_{ij} dv_g=O([\e_j^n]^2)$ in $S\sm \cup_{j=1}^m
B_{r}(\xi^n_j)$. Thus, we get
$$\Delta_g\phi_n(x)=\sum_{j=1}^kO((\e^n_j)^2)+c_0+o(1)\qquad\text{uniformly for}\quad x\in S\sm \cup_{j=1}^m B_{r}(\xi^n_j).$$
Therefore, passing to a subsequence $\phi_n\to\phi^*$ as
$n\to+\infty$ in $C^{1}$ sense over compact subsets of
$S\sm\{\xi^*_1,\dots,\xi^*_m\}$. Since $|\phi^*(x)|\le1$ for all
$x\in S\sm\{\xi^*_1,\dots,\xi^*_m\}$, it follows that $\phi^*$ can
be extended continuously to $S$ so that $\phi^*$ satisfies
$\Delta_g \phi^*=c_0$, in $S$ and $\int_S\phi^*\,dv_g=0$ using
dominated convergence. By, $\int_S \Delta_g\phi^*=0$ we get that
$c_0=0$. Therefore, $\phi^*\equiv0$, and the claim follows.
\end{proof}

\begin{claim}
For all $i=0,1,2$, $j=1,\dots,k$ it holds that $c_{ij}^n\to 0$ as $n\to+\infty$.
\end{claim}
\begin{proof}[\dem]
The claim readily follows for $i=1,2$  and $j=1,\dots,k$ from
Claim \ref{estcij} and the estimate in \eqref{estcij0} in view of $\|h\|_*=o(1)$. Let us
refine the estimate for $c_{0j}$, $j=1,\dots,k$. It is clear that in $B_{r_0}(\xi_j)$
$$\lap_g\phi+e^{-\varphi_j}\e_j^{-2}e^{w_j}\phi+c(\phi) =h+\sum_{p=0}^2c_{p j}\lap_g
Z_{p j}-\sum_{p=0}^2\sum_{q=1}^k
{c_{pq}\over|S|}\int_S\chi_q\lap_g
Z_{pq}\,dv_g.$$ Hence,
integrating on $B_{r}(\xi_j)$ with $0<r<r_0$ we find that as
$n\to+\infty$
\begin{equation*}
\begin{split}
\int_{B_r(\xi_j)}e^{-\varphi_j}\e_j^{-2}e^{w_j}\phi\,dv_g=&\,-\int_{B_r(\xi_j)}\lap_g\phi\,dv_g-c(\phi)|B_r(\xi_j)|
+\int_{B_r(\xi_j)}h\,dv_g\\
&\,+\sum_{p=0}^2c_{p j}\int_{B_r(\xi_j)}\lap_g Z_{p
j}\,dv_g-\sum_{p=0}^2\sum_{q=1}^k
{c_{pq}\over|S|}|B_r(\xi_j)|\int_S\chi_p\lap_g
Z_{pq}\,dv_g\\
=&\,-\int_{\fr B_r(0)}\fr_\nu(\phi\circ
y_{\xi_j}^{-1})\,d\sigma+O\bigg(|c(\phi)|+\|h\|_*+\sum_{q=1}^k
\e_q^2\bigg)=o(1),
\end{split}
\end{equation*}
in view of Claims \ref{estcij} and \ref{ab} (the convergence $\phi\to0$ in $C^1$ sense). Thus, we obtain that
$$c_{0j}=O\bigg(\bigg|\int_{B_r(\xi_j)}e^{-\varphi_j}\e_j^{-2}e^{w_j}\phi\,dv_g\bigg|+\e_j^2|\log\e_j|\,
\|\phi\|_\infty+\|h\|_*+\e_j\sum_{p=0}^2\sum_{q=1}^k
|c_{pq}|\bigg)=o(1)$$ and the claim follows.
\end{proof}

\noindent We follow ideas shows in \cite{DeKM} to prove an estimate for
$\phi_n$. For $y\in B_{r_0}(0)$, define
$\hat\phi_{n,j}(y)=\phi_n(y_{\xi_j}^{-1}(y))$ and
$$\hat
h_n(y)=e^{\hat\varphi_j(y)}\Big[-c(\phi_n)+h_n(y_{\xi_j^n}^{-1}(y))+\sum_{p=0}^2\sum_{q=1}^kc_{pq}\lap_gPZ_{pq}^n(y_{\xi_j}^{-1}(y))\Big],$$
so that,
$$\hat L_{j}(\hat\phi_{n,j}):=\Delta \hat\phi_{n,j} + {8\mu_j^2\e_j^2 \over (\mu_j^2\e_j^2 + |y|^2 )^2}
\hat\phi_{n,j}=\hat h_n.$$
Let us fix such a number $\de>0$ which we may take
larger whenever it is needed and a small $0<r<r_0$. Now, let us
consider the ``annulus norm'' and ``boundary annulus norm''
$$\|\phi\|_a =  \|\phi\|_{L^\infty(\Om_{\de,r})}\qquad\text{and}\qquad \|\phi\|_b =  \|\phi\|_{L^\infty(\fr\Om_{\de,r})},$$
where $\Omega_{\de,r} :=  B_{r}(0) \sm \bar B_{\de\e_j}(0)$. Note that $\fr\Om_{\de,r}=\cup_{j=1}^m [\fr B_{\de\e_j}(0)\cup
\fr B_{r}(0)]$. By now it is rather standard to show that for functions on $\Om_{\de,r}$ the operator $\hat L_j$ satisfies the maximum principle in $\Omega_{\de,r}$ for $\de$ large and $r>0$ small enough, see for example \cite{DeKM}. In fact, the function $g(y):=Y_{0j}(a\e_j^{-1}y)=2{a^2|y|^2-\mu_j^2\e_j^2\over \mu_j^2\e_j^2+a^2|y|^2}$ with $0<4a<1$ and $\sqrt{5\over3}a<\de$ satisfies $\hat L_j(g)<0$ and $g>{1\over 2}>0$ in $\Om_{\de,r}$. As a consequence, we get that

\begin{claim}
There is a constant $C>0$ such that if $ \hat L_j(\phi) = h$ in
$\Omega_{\de,r}$ then
\begin{equation} \label{ee1}
\|\phi\|_a \le C[ \|\phi\|_b + \| h\|_{**} ],
\end{equation}
where
$$\|h\|_{**}=\sup_{y\in B_r(0)}{|h(y)|\over \rho_{\mu,\e,m}(y)+1}.$$
\end{claim}

\begin{proof}[\dem]
We shall omit the subscript $n$ in the quantities involved.
We will establish this inequality with the use of suitable
barriers. Consider now the solution of the problem
$$-\Delta \psi = 1,\qquad\text{ for $\de\e_j<|y| < r,$}$$
 $$\psi(y) = 0\quad\text{ for $|y| = \de\e_j,$ $|y| =r$.}$$ Direct computation shows that
$$
\psi_{j,1}(y)=\frac{\de^2\e_j^2-r^2}{4}+\frac{\de^2\e_j^2-r^2}{4}\frac{\log\frac{|y|}{\de\e_j}}{\log\frac{\de\e_j}{r}},
$$
Note that $\ds0\le
\psi_{j,1}\le \frac{r^2-\de^2\e_j^2}{4}\le {r^2\over 4}$ hence these functions
$\psi_{j,1}$ have a uniform bound independent of $\e_j$.

On the other hand, let us consider the function $g$ defined above, and let us
set
$$
\psi (y) = 2\|\phi\|_b\, g(y) + \|h\|_{**}[\psi_{j,1}(y)+\psi_{j,2}(y)],
$$
where $\psi_{j,2}$ is the solution to
$$
\begin{array}{l}
-\Delta \psi =\ds 2\rho_{\mu_j,\e_j,m_j}, \quad
\de\e_j<|y| < r, \cr\cr \psi (y) = 0 \hbox{ for }|y| = \de\e,\quad |y| =r.
 \end{array}$$
Then, it is easily checked that, choosing $\de$ larger if necessary,
$\hat L_j(\psi) \le h$ and $\psi \ge |\phi|$ on $\partial
\Omega_{\de,r}$. Hence $|\phi| \le \psi$ in $\Omega_{\de,r}$.
In fact, we have that for all $y\in\fr\Om_{\de,r}$
$$\psi(y)\ge 2\|\phi\|_b\,g(y)\ge \|\phi\|_b\ge |\phi(y)|.$$
Also, we have that choosing $2C(\de,r)+r^2<1$ (for $\de$
large enough and $r$ small enough)
\begin{equation*}
\begin{split}
\hat L_j(\psi)
&<\|h\|_{**}\lf(\lap [\psi_{j,1}+\psi_{j,2}]+{8\mu_j^2\e_j^2 \over (\mu_j^2\e_j^2 + |y|^2 )^2}[\psi_{j,1}+\psi_{j,2}]\rg) \\&
\le\|h\|_{**}\lf(-1-2\rho_{\mu_j,\e_j,m_j}(y)+{8\mu_j^2\e_j^2\over (\mu_j^2\e_j^2+|y|^2)^2}[2C(\de,r)+r^2]\rg)\\
&\le-\|h\|_{**}[\rho_{\mu_j,\e_j,m_j}+1]
\le h,
\end{split}
\end{equation*}
in view of $\rho_{\mu_j,\e_j,m_j}(y)\ge{8\mu_j^2\e_j^2\over (\mu_j^2\e_j^2+|y|^2)^2}$
Hence, we conclude that $|\phi(y)|\le \psi(y)$ for all
$\de\e_j<|y|<r$, for every $j=1,\dots,k$ and the claim
follows.
\end{proof}

\medskip

The following intermediate result provides another estimate.
Again, for notational simplicity \emph{we omit} the subscript $n$
in the quantities involved.

\begin{claim}\label{LpLe3a}
There exist constants $C>0$ such that for large $n$
\begin{equation}\label{ei1}
\|\phi\|_{L^\infty(\cup_{j=1}^mB_{r_0}(\xi_j))}\leq
C\, \big\{\|\phi\|_{L^\infty(\cup_{j=1}^mB_{\de\e_j}(\xi_j))} + o(1)
\big\}.
\end{equation}

\end{claim}

\begin{proof}[\dem]First, note that
$$
\|\hat h\|_{**}\le C\bigg[|c(\phi)|+\|h\|_*+\sum_{p=0}^2\sum_{q=1}^k|c_{pq}|\bigg]$$
in view of the definition of $\|\cdot\|_*$ and $\|\cdot\|_{**}$ and $\|\lap_gPZ_{pq}\|_*\le C$. From estimate \eqref{ee1} we
deduce that there is a constant $C>0$ \st
\begin{equation}\label{LpLe3aa}
\begin{split}
\|\phi\|_{L^{\infty}(B_{r}(\xi_j) \sm \bar B_{\de\e_j}(\xi_j))}&=\|\hat\phi\|_a \le C\lf[ \|\hat\phi\|_b + \|\hat h\|_{**}\rg]\\
&\le C\bigg[\|\phi\|_{L^\infty(\fr[B_{r}(\xi_j) \sm \bar B_{\de\e_j}(\xi_j)])}+|c(\phi)|+\|h\|_*+\sum_{p=0}^2\sum_{q=1}^k|c_{pq}|\bigg]
\end{split}
\end{equation}
From \eqref{cpe}
we find that for large $n$
\begin{equation}\label{cpe1}
\|\phi\|_{L^\infty(S\sm\cup_{l=1}^m B_r(\xi_j))}=o(1).
\end{equation}
Furthermore, we have that $c(\phi)=o(1)$, since $c_0=0$. By the
assumption, we know that $ \|h\|_*=o(1)$. Now, from
\eqref{LpLe3aa} it is clear that
\begin{equation*}
\begin{split}
\|\phi\|_{L^\infty(\cup_{j=1}^mB_r(\xi_j))}
& \le \|\phi\|_{L^\infty(\cup_{j=1}^mB_{\de\e_j}(\xi_j))}+C\lf[\|\hat\phi\|_b+\|\hat h\|_*\rg]\\
\end{split}
\end{equation*}
and the conclusion follows by \eqref{cpe1}.
\end{proof}

\noindent We continue with the proof of Proposition \ref{p2} and
we get the following fact.

\begin{claim}
There exists an index $j\in\{1,\dots,k\}$ \st passing to a
subsequence if necessary,
\begin{equation}\label{inf1}
\liminf_{n\to\infty} \|\phi_n\|_{L^\infty(B_{\de\e_j}(\xi^n_j))}
\geq\kappa>0.
\end{equation}
\end{claim}

\begin{proof}[\dem]
Arguing by contradiction, if for all $j=1,\dots,k$
\begin{equation*}
\liminf_{n\to\infty} \|\phi_n\|_{L^\infty(B_{\de\e_j}(\xi^n_j))}=0,
\end{equation*}
then \eqref{ei1} and \eqref{cpe1} implies that, passing to a
subsequence if necessary, $\|\phi_n \|_\infty\to 0$ as
$n\to+\infty$. On the other hand, we know that $\|\phi\|_\infty=1$
for all $n\in\N$. This conclude \eqref{inf1}.
\end{proof}

\noindent Let us set $
\psi_{n,j}(z)=\phi_{n}\big(y_{\xi^{n}_j}^{-1}(\e_j^n z)\big)=\hat\phi_{n,j}(\e_j^nz)$ for any $j$ and $z\in B_{r_0/\e_j^n}(0)$. We notice that $\psi_{n,j}$ satisfies
$$\Delta \psi_{n,j} +{8[\mu_j^n]^2\over([\mu_j^n]^2+|z|^2)^2} \,
\psi_{n,j} = [\e_j^n]^2\hat h_{n,j}(\e_j^nz)
\quad\hbox{ in } B_{r_0/\e_j^n}(0).$$ Elliptic
estimates and \eqref{inf1} readily imply that $\psi_{n,j}$
converges uniformly over compact subsets of $\R^{2}$ to a bounded,
non-zero solution $\psi^*_j$ of
\[
\Delta \psi+ \frac{8\mu^{2}}{(\mu_j^{2}+|z|^{2})^{2}}\psi=0,\qquad \mu_j=\mu_j^*
\]
in view of $\big|[\e_j^n]^2\hat h_{n,j}(\e_j^nz)\big|\le C\|h_n\|_*$ for $z$ in compact subsets of $\R^{2}$. This implies that $\psi^*_j$ is a linear combination of the
functions $Y_{ij}$, $i=0,1,2$. Thus, we have that for some
constants $a_{ij}$, $i=0,1,2$,
$\psi^*_j=a_{0j}Y_{0j}+a_{1j}Y_{1j}+a_{2j}Y_{2j}$. See \cite{bp} for a proof. But, from
\eqref{ephin}, orthogonality conditions over $\psi_{n,j}$ pass to
the limit thanks to $\|\psi_{n,j}\|_{\infty}\leq C$ and dominated
convergence, namely,
$$\int_{\R^2}\lap Y_{ij}\,\psi^*_j=0,\qquad\text{for}\quad i=0,1,2.$$
A contradiction with \eqref{inf1} arises since this implies that $a_{0j}=a_{1j}=a_{2j}=0$. Thus, we get the estimate $\|\phi\|_\infty\le C\|h\|_*$. Hence, from the same argument shown in the proof of the Claim \ref{estcij}, we conclude the estimates \eqref{estmfe1}.

\medskip

Now, let us prove the solvability assertion. To this
purpose we consider the space
$$
H= \biggl\{ \phi \in H_0^1(S) \ : \ \int_{S} \lap_gPZ_{ij} \,\phi
\, = \, 0 \quad\hbox{ for }  \; i=0,1,2,\,j=1,\dots,k\biggr\},
$$
endowed with the usual inner product $\ds [\phi , \psi ] = \int_{S}
\langle \nabla \phi,  \nabla\psi \rangle_g\, dv_g$. Problem \eqref{plco} expressed in weak
form is equivalent to that of finding a $\phi \in H$, such that
$$
[\phi , \psi ] = \int_{S}  [ K\phi - h ]\, \psi\; dv_g,
\quad\hbox{ for all }  \,\psi \in H.
$$
Recall that $\ds \int_{S}h\,dv_g=0$. With the aid of Riesz's
representation theorem, this equation gets rewritten in $H$ in the
operator form $\phi = \ml{K}(\phi) + \tilde h$, for certain $
\tilde h \in H$, where $\ml{K}$ is a compact operator in $H$.
Fredholm's alternative guarantees unique solvability of this
problem for any $h$ provided that the homogeneous equation $ \phi
= \ml{K}(\phi) $ has only the zero solution in $H$. This last
equation is equivalent to \eqref{plco} with $h\equiv 0$. Thus
existence of a unique solution follows from the a priori estimate
\eqref{estmfe1}.

\medskip
We have just proven that the unique
solution $\phi = T_\la(h)$ of \eqref{plco} defines a continuous linear
map from the Banach space ${\ml C}_*$ of all functions $h$ in
$L^\infty(S) $ for which $\|h\|_* <+\infty$ and $\ds \int_{S}h=0$,
into $L^\infty$, with bounded norm.

\medskip
It is important to understand the differentiability of the operator $T$ with respect to the variable either
$\beta=\xi_j$ or $\beta=m_j$. Fix $h\in {\ml C}_{*}$ and let
 $\phi = T_\la(h)$. Let us recall that $\phi$ satisfies \eqref{plco}, for some
(uniquely determined) constants $c_{ij}$, $i=0,1,2$, $j=1,\dots,k$.
We want to compute derivatives of $\phi$ with respect to the
parameters $\beta=\xi_{l}$ or $\beta=m_l$. Formally $X=
\partial_{\beta} \phi$ should satisfy
$$
{L}(X) =  -\partial_{\beta}K \, \phi+ {1\over
|S|}\int_{S} \partial_{\beta}K\phi+
\sum_{i=0}^2\sum_{j=1}^kc_{ij}\,
\partial_{\beta}(\lap_gP Z_{ij}) \,+ \sum_{i=0}^2\sum_{j=1}^k d_{ij}\lap_gPZ_{ij}\,,
$$
where (still formally) $d_{ij} = \partial_{\beta}(c_{ij})$,
$i=0,1,2$, $j=1,\dots,k$. The orthogonality conditions now become
\begin{eqnarray*}
\int_{S } \lap_gPZ_{ij} X =-\int_{S}
\fr_{\beta}(\lap_gPZ_{ij})\phi, \qquad i=0,1,2,j=1,\dots,k.
\end{eqnarray*}
We will recast $X$ as follows. Consider the $Y=X+\sum_{i,j}b_{ij} PZ_{ij}$, where the coefficients $b_{ij}$ are chosen so that $Y$ satisfies the orthogonality conditions $\ds \int_S Y\lap_g PZ_{ij}=0$ for all $i,j$. The coefficients $b_{ij}$ are well-defined since they satisfy an almost diagonal system in view of \eqref{ilpzpz}. Furthermore, it holds that for $\beta=\xi_l$, $|b_{ij}|\le C \|h\|_*$ if $j\ne l$ and $|b_{il}|\le \ds{C\over \e_l} \|h\|_*$; and for $\beta=m_l$, $|b_{ij}|\le C \|h\|_*$ if $j\ne l$ and $|b_{il}|\le C|\log\e_l | \|h\|_*$, in view of $\|\fr_{\xi_{l}}(\lap_g PZ_{ij})\|_*\le C $ if $j\ne l$ and $\|\fr_{\xi_{l}}(\lap_g PZ_{il})\|_* \le \ds{C\over \e_l} $; and $ \|\fr_{m_{l}}(\lap_g PZ_{ij})\|_* \le C $ if $j\ne l$ and $\|\fr_{m_{l}}(\lap_g PZ_{il})\|_* \le C|\log\e_l |$. Then the function $X$ above can be uniquely expressed as
$$
X = T_\la (f) - \sum_{i=0}^2\sum_{j=1}^kb_{ij}\,PZ_{ij}.
$$
where the function
$$
f:= -
\partial_{\beta} K\, \phi + {1\over | S |}\int_{S} \partial_{\beta } K\phi+ \sum_{i=0}^2\sum_{j=1}^k\Big[b_{ij}\, {L}(PZ_{ij}) + c_{ij}\,\partial_{\beta }(\lap_g PZ_{ij})\Big] . 
$$
This computation is not just formal. Arguing directly by
definition it shows that indeed $\partial_{\beta}\phi = X$.
Also, we find that $\ds \|f\|_*\le {C\over \e_l}  \|h\|_*$, for $\beta=\xi_l$ and $\|f\|_*\le C |\log \e_l| \|h\|_*$, for $\beta=m_l$, in view of
$$
\|f\|_*\le \|\partial_{\beta} K\|_*\, \|\phi\|_\infty \bigg[1+ \frac{1}{|S|}\int_{S}
\rho(x)\,dv_g\bigg]+ \sum_{i=0}^2\sum_{j=1}^k\Big[ |b_{ij}|\, \|{L}(PZ_{ij})\|_* + |c_{ij}|\,\|\partial_{\beta }(\lap_g PZ_{ij})\|_*\Big]  .
$$
Indeed, it is easy to check that $\ds \int_{S}
\rho(x)\,dv_g\le C$, $\| L( PZ_{ij})\|_*\le
C\e_j $ for $i=1,2$ and $\| L( PZ_{0j})\|_*\le C $. Furthermore, from the definition of $K$ it follows that $\|\fr_{\xi_{l}}K \|_*\le \ds{C\over \e_l}$ and $\lf\| \fr_{ m_l }K \rg\|_* \le C |\log \e_l| $. Moreover,
using estimate \eqref{estmfe1} applied with R.H.S. $f$, we find that
\begin{equation*}
\begin{split}
\|\fr_{\xi_l }\phi\|_\infty &\le
C \bigg[\|f\|_*+\sum_{i=0}^2\sum_{j=1}^k |b_{ij}| \bigg],
\end{split}
\end{equation*}
so that, $\ds\|\fr_{\beta }\phi\|_*\le {C\over \e_l}  \|h\|_*$ and $\|\fr_{m_{l} }\phi\|_\infty\le C |\log \e_l | \|h\|_*$.
Finally, we conclude that
\begin{equation}\label{difx}
\| \partial_{(\xi_{l})_i} T_\la(h)\|_\infty  \le {C\over
\e_l }\, \|h\|_*\,\quad \mbox{for }\ i=1,2, l=1,\dots, k
\end{equation}
and
\begin{equation}\label{difm}
\| \partial_{m_{l}} T_\la(h)\|_\infty  \le  C\,|\log \e_l| \, \|h\|_*\ \quad \mbox{for }\ l=1,\dots, k .
\end{equation}
This finishes the proof of proposition \ref{p2}.
\end{proof}


\section{\hspace{-0.5cm}: The nonlinear problem} \label{appeB}

\noindent By Proposition \ref{p2} we now deduce the following.

\begin{proof}[{\bf Proof (of Proposition \ref{lpnlabis})}] First, note that $R\in L^\infty(S)$, $\|R\|_*<+\infty$, $\ds\int_{S} R=0$ and \linebreak $\ds \int_{S}N(\phi)=0$ for any $\phi\in C(S)$. Next, we
observe that in terms of the operator $T$ defined in Proposition
\ref{p2}, the latter problem becomes
\begin{equation}\label{ptfix}
\phi = - T\lf( R + N(\phi)\rg):=\ml{A}(\phi).
\end{equation}
For a given number $\nu>0$, let us consider
$$
\ml{F}_\nu = \{\phi\in C(S) : \| \phi \|_\infty \le
\nu \la\}
$$
From the Proposition \ref{p2}, we get
\begin{equation*}
\begin{split}
\|\ml{A}(\phi)\|_\infty & \le C \lf[ \|R \|_*+
\|N(\phi)\|_*\rg].
\end{split}
\end{equation*}
From \eqref{re} we know that $\| R \|_*\le C\la$. Furthermore, it follows that for certain $0<s<s^*<1$
$$f(V+\phi)-f(V)-f'(V)\phi=\int_0^1[f'(V+t\phi)-f'(V)]\,dt\, \phi=[f'(V+s^*\phi)-f'(V)]\phi=f''(V+s\phi)\,s^*\phi^2.$$
From the definition of $f$ in \eqref{effe} and the estimates used to prove Lemma \ref{estrr0}, it follows that
$$f''(V+s\phi)=2\la^2(V+s\phi)e^{\la(V+s\phi)^2}\lf[3+2\la(V+\phi)^2\rg]=\la Ve^{\la V^2}O(1)+\la^2e^{\la V^2}O(1)$$
so that, from the definition of $\|\cdot\|_*$ we obtain that $\|f''(V+s\phi)\|_*\le C$. Thus, we find that
\begin{equation*}
\begin{split}
\| N (\phi) \|_* & \le \Big[\|f(V+\phi)-f(V)-f'(V)\phi\|_*+\|f'(V)-K\|_*\|\phi\|_\infty\Big]\bigg[1+ \frac{1}{|S|}\int_{S}
\rho(x)\,dv_g\bigg]\\
&\le C \big[\|\phi\|_\infty^2+\la\|\phi\|_\infty\big]\le C\la^2[\nu^2+\nu],
\end{split}
\end{equation*}
in view of \eqref{k} and $\ds \int_{S}
\rho(x)\,dv_g\le C$. Hence, we get for any $\phi\in \ml{F}_\nu$,
\begin{equation*}
\|\ml{A}(\phi)\|_\infty \le C \la \lf[1 + (\nu+\nu^2)\la\rg].
\end{equation*}

On the other hand, for $\phi_1$ and $\phi_2$ and certain $0<s,t^*<1$ we have that
\begin{equation*}
\begin{split}
f(V+\phi_1)&-f(V+\phi_2)-f'(V)(\phi_1-\phi_2)\\ & =\int_0^1[f'\big(V+\phi_2+t\{\phi_1-\phi_2\}\big)-f'(V)]\,dt\, [\phi_1-\phi_2]\\
&=\lf[f'\big(V+t^*\phi_1+\{1-t^*\}\phi_2\big)-f'(V)]\, [\phi_1-\phi_2\rg]\\
&=f''\big(V+s[t^*\phi_1+\{1-t^*\}\phi_2]\big)\,[t^*\phi_1+\{1-t^*\}\phi_2]\,[\phi_1-\phi_2],
\end{split}
\end{equation*}
so that,
$$\|f(V+\phi_1)-f(V+\phi_2)-f'(V)(\phi_1-\phi_2)\|_*\le C\,[\|\phi_1\|_\infty+\|\phi_2\|_\infty]\,\|\phi_1-\phi_2\|_\infty,$$
in view of $\|f''\big(V+s[t^*\phi_1+\{1-t^*\}\phi_2]\big)\|_*\le C$. Hence, given any $\phi_1,\phi_2\in\ml{F}_\nu$, we have that
\begin{equation*}
\begin{split}
\|N(\phi_1)-N(\phi_2)\|_* & \le C (\|\phi_1\|_\infty +
\|\phi_2\|_\infty) \|\phi_1-\phi_2\|_\infty+C\la \|
\phi_1-\phi_2\|_\infty\\
& \le C \la[\nu+1]\,\|\phi_1-\phi_2\|_\infty
\end{split}
\end{equation*}
with $C$ independent of $\nu$. Therefore, from the Proposition \ref{p2}
\begin{equation*}
\begin{split}
\|\ml{A}(\phi_1)-\ml{A}(\phi_2)\|_\infty & \le C\|N(\phi_1)-N(\phi_2)\|_*\le C \la[\nu+1]\,\|\phi_1-\phi_2\|_\infty\\
\end{split}
\end{equation*}
It follows that for all $\e$ sufficiently small $\ml{A}$ is a
contraction mapping of $\ml{F}_\nu$ (for $\nu$ large enough), and
therefore a unique fixed point of $\ml{A}$ exists in $\ml{F}_\nu$.
\medskip

Let us now discuss the differentiability of $\phi$ depending on
$(\xi,m)$, i.e., $(\xi,m)\mapsto \phi(\xi,m)\in C(S)$ is $C^1$.
Since $R$ depends continuously (in the $*$-norm) on $(\xi,m)$, using
the fixed point characterization \eqref{ptfix}, we deduce that the
mapping $(\xi,m)\mapsto \phi$ is also continuous. Then, formally
\begin{equation*}
\begin{split}
\fr_{\beta } N(\phi)  = &\,[f'(V+\phi)-f'(V)-f''(V)\phi ] \fr_\be V+[f'(V+\phi)-f'(V)]\fr_\be\phi\\
&-{1\over |S|} \int_{S}\big\{ [f'(V+\phi)-f'(V)-f''(V)\phi ] \fr_\be V+[f'(V+\phi)-f'(V)]\fr_\be\phi \big\} \\
&+[f''(V)\fr_\be V-\partial_{\be } K]\phi + [f'(V)-K] \fr_\be\phi\\
 & \, -{1\over |S|} \int_{S}\big( [f''(V)\fr_\be V-\partial_{\be } K]\phi + [f'(V)-K] \fr_\be\phi \big).
\end{split}
\end{equation*}
so that, we estimate as follows
\begin{equation*}
\begin{split}
\|\fr_{\be } N(\phi)\|_* & \le C \Big[ \|f''(V+s'\phi) - f''(V)\|_* \ \|\phi\|_\infty  \|\fr_\be V\|_\infty + \|f''(V+s'\phi)\|_*\|\phi\|_\infty \|\fr_\be\phi\|_\infty \\
&\ \quad\quad +\|f''(V)\fr_\be V-\partial_{\be } K]\|_*\|\phi\|_\infty + \|[f'(V)-K]\|_* \|\fr_\be\phi\|_\infty \Big],
\end{split}
\end{equation*}
for some $s'\in (0,1)$. In particular, precisely for $\be=\xi_l$ we obtain
\begin{equation*}
\begin{split}
\|\fr_{\xi_l } N(\phi)\|_* &\le C\Big[ \|\phi\|_\infty^2\|\fr_{\xi_l} V\|_\infty + \|\phi\|_\infty \|\fr_{\xi_l}\phi\|_\infty+{\la\over\e_l}\|\phi\|_\infty + \la \|\fr_{\xi_l} \phi\|_\infty\Big]\\
 & \le C \Big[ {\la^2\over \e_l} + \la \|\fr_{\xi_l}\phi\|_\infty \Big],
\end{split}
\end{equation*}
in view of $\|f''(V+s'\phi) - f''(V)\|_*\le C\|\phi\|_\infty$, $\ds \| \fr_{\xi_l} V \|_\infty \le {C\over \e_l}$, $\ds\|f''(V)\fr_{\xi_l}V - \fr_{\xi_l}K\|_*\le {C\la\over\e_l}$ and estimate \eqref{k}, and for $\be=m_l$ we obtain
\begin{equation*}
\begin{split}
\|\fr_{m_l } N(\phi)\|_* &\le C\Big[ \|\phi\|_\infty^2\|\fr_{m_l} V\|_\infty + \|\phi\|_\infty \|\fr_{m_l}\phi\|_\infty+\la |\log \e_l|\, \|\phi\|_\infty + \la \|\fr_{m_l} \phi\|_\infty\Big]\\
 & \le C \Big[ \la^2|\log \e_l | + \la \|\fr_{m_l}\phi\|_\infty \Big],
\end{split}
\end{equation*}
in view of $\| \fr_{m_l} V \|_\infty \le C|\log \e_l |$, $\|f''(V)\fr_{m_l}V - \fr_{m_l}K\|_*\le C\la|\log\e_l|\le C$. Also, observe that we have
$$\fr_{\be }\phi= - (\fr_{\be }T)\lf( R + N (\phi)\rg) - T\lf(\fr_{\be }\lf[ R + N (\phi)\rg]\rg).$$
So, using \eqref{difx} and previous estimates, we get
\begin{equation*}
\begin{split}
\|\fr_{\xi_l }\phi\|_\infty  &\le {C\over\e_l} \, \|R + N(\phi)\|_* + C \,
\|\fr_{\xi_{l}}(R + N (\phi))\|_*\\
& \le {C\over \e_l} \,\lf[ \la +
\la \| \phi\|_\infty+ \|\phi\|_\infty^2\rg] +C\lf[ {\la \over \e_l} +  {\la^2\over \e_l} + \la \|\fr_{\xi_l}\phi\|_\infty \rg]\\
& \le {C\la \over \e_l}  + C \la  \|\fr_{\xi_l}\phi\|_\infty .
\end{split}
\end{equation*}
and similarly, using \eqref{difm} and previous estimates, we get
\begin{equation*}
\begin{split}
\|\fr_{m_l }\phi\|_\infty  &\le C |\log \e_l |
\, \|R + N(\phi)\|_* + C \, \|\fr_{m_{l}}(R + N (\phi))\|_*\\
& \le C |\log \e_l | \,\lf[ \la +
\la \| \phi\|_\infty+ \|\phi\|_\infty^2\rg] +C\lf[ \la |\log \e_l | +  \la^2 |\log \e_l | + \la \|\fr_{\xi_l}\phi\|_\infty \rg]\\
& \le C\la |\log \e_l|  + C \la  \|\fr_{\xi_l}\phi\|_\infty .
\end{split}
\end{equation*}
We have used an estimate for $\|\fr_{\be } R\|_*$. From the definition of $V$, the definition of $\fr_{\be } R$
$$\fr_{\be} R(y)=\lap_g \fr_{\be }V(y)+ f'(V) \fr_{\be }V-{1\over|S|} \int_{S}f'(V) \fr_{\be }V \,dv_g,$$
similar computations to deduce \eqref{re} and from the definition of *-norm it follows that
$$\|\fr_{\xi_{l}} R\|_*\le \dfrac{C\la}{\e_l}\qquad\text{ and }\qquad\|\fr_{m_{l}} R\|_*\le C\la|\log \e_l |.$$ Thus, we conclude \eqref{cotadphi}.

The above computations can be made rigorous by using the implicit
function theorem and the fixed point representation \eqref{ptfix}
which guarantees $C^1$ regularity in $(\xi,m)$.
\end{proof}

\section{\hspace{-0.5cm}: Proof of Lemma \ref{cpfc0bis} } \label{appeC}

\begin{proof}[\dem]
Let us differentiate the function $\ml{F}_\la(\xi,m)$ with respect to either
$\be=(\xi_l)_q$ or $\be=m_l$, with $q=1,2$ and $l=1,\dots,k$. Since $\sqrt\la \,V(\xi,m) =U(\xi,m)$ and $\sqrt\la\,\phi(\xi,m) = \ti\phi(\xi,m)$, we can differentiate directly $J_\la\big(\sqrt\la[V+\phi] \big)$ (under the integral sign), so that integrating by parts we get
\begin{equation*}
\begin{split}
\fr_{\be }\ml{F}_\la(\xi,m) = &\,\sqrt\la\,
DJ_\la\lf(\sqrt{\la}[V+\phi]\rg)\lf[ \fr_{\be}V
+ \fr_{\be }\phi\rg]\\
=&\,-\la\int_S\lf[\Delta_g(V+\phi) + \la( V+\phi) e^{\la(V+\phi)^2}\rg]\,\lf[\fr_{\be }V + \fr_{\be }\phi\rg]\,dv_g\\
=&\, -\la \sum_{i=0}^{2}
\sum_{j=1}^k c_{ij} \int_{S} \Delta_gPZ_{ij}\,\lf[\fr_{\be }V + \fr_{\be }\phi\rg],
\end{split}
\end{equation*}
since $\ds \int_{S} \big(\fr_{\be }V + \fr_{\be }\phi\big)=0$. From the result of \ref{lpnlabis}, this expression defines a continuous function of $(\xi,m)$. Let us stress that $\mu_j=\mu_j(\xi,m)$ and $\e_j=\e_j(m_j)$. Hence, from \eqref{dxiv} we have that uniformly in $S$
\begin{equation*}
\begin{split}
\fr_{(\xi_l)_q}V(x)
=&\,- {m_l 	\over \mu_l\e_l} \chi_l {2\mu_l\e_l\fr_{(\xi_l)_q}(|y_{\xi_l}(x)|^2)\over \mu_l^2\e_l^2+|y_{\xi_l}(x)|^2} +O(1)\\
\end{split}
\end{equation*}
and from \eqref{del} and \eqref{dmv}
\begin{equation*}
\begin{split}
\fr_{m_l }V(x)
=&\, 	 -2 \chi_l\log(\mu_l^2\e_l^2+|y_{\xi_l}(x)|^2) + \chi_l {8\mu_l^2\e_l^2\log\e_l \over \mu_l^2\e_l^2+|y_{\xi_l}(x)|^2} +O(1),
\end{split}
\end{equation*}
in view of $ \e_l\fr_{m_l} \e_l=\ds -{2\e_l^2\over m_l}\log\e_l+O(1)$. Let us assume that $D_\xi \ml{F}_\la(\xi,m)=0$ and $D_m\ml{F}_\la(\xi,m)=0$. Then,
from the latter equality and the estimates \eqref{cotadphi} we get
$$
\sum_{i=0}^{2} \sum_{j=1}^k c_{ij} \int_{S} \lap_gPZ_{ij}\,\lf[\e_l\fr_{(\xi_{l})q}V + O(\la)\rg]=0,\qquad q=0,1,2,\;l=1,\dots,k
$$
and also,
$$
\sum_{i=0}^{2} \sum_{j=1}^k c_{ij} \int_{S} \lap_gPZ_{ij}\,\lf[{\fr_{m_l}V\over\log\e_l} + O(\la)\rg]=0,\qquad \;l=1,\dots,k.
$$
Using
$$\e_l\fr_{(\xi_{l})_q}V = -\ds{m_l\over \mu_l} \chi_l {2\mu_l\e_l\fr_{(\xi_l)_q}(|y_{\xi_l}(x)|^2)\over \mu_l^2\e_l^2+|y_{\xi_l}(x)|^2} + O(\e_l)$$
and
$$\dfrac{\fr_{m_{l} }V}{\log\e_l} = \ds{1\over \log \e_l} \chi_l[U_l-\log(8\mu_l^2\e_l^2)] +2\chi_l[Z_{0l}+2] + O\Big(\ds{1\over|\log\e_l|}\Big),$$
where $O(\e_l)$ and $O(|\log \e_l|^{-1})$ are in the $L^\infty$ norm as $\la\to 0$, it follows
$$
\sum_{i=0}^{2} \sum_{j=1}^m c_{ij} \int_{S} \lap_g PZ_{ij}\,\lf[\chi_l {2\mu_l\e_l\fr_{(\xi_l)_q}(|y_{\xi_l}(x)|^2)\over \mu_l^2\e_l^2+|y_{\xi_l}(x)|^2} + o(1)\rg]=0,\qquad q=0,1,2,\; l=1,\dots,m.
$$
$$
\sum_{i=0}^{2} \sum_{j=1}^m c_{ij} \int_{S} \lap_g PZ_{ij}\,\lf[\chi_l(Z_{0l} +2) + o(1)\rg]=0,\qquad l=1,\dots,m.
$$
with $o(1)$ small in the sense of the $L^\infty$ norm as $\la\to0$. The above system is diagonal dominant and we thus get $c_{ij}=0$ for $i=0,1,2$, $j=1,\dots,k$. We have used that
$$\int_{S} \lap_g PZ_{ij}\,\chi_l[U_l-\log(8\mu_l^2\e_l^2)]=O(1).$$
The proof of Lemma \ref{cpfc0bis} is finished.
\end{proof}

\end{appendices}





\end{document}